\documentclass[10pt]{article}
\usepackage{amsmath,amsfonts,amsthm,amssymb,graphicx,float,tikz,indentfirst,graphics,setspace,young,mathrsfs,bbm}
\usepackage[hidelinks]{hyperref}
\usepackage[enableskew,vcentermath]{youngtab}
\usepackage[all]{xy}
\usetikzlibrary{decorations.pathreplacing,patterns,snakes}
\theoremstyle{plain}
\newtheorem{theorem}{Theorem}[section]
\newtheorem{conjecture}[theorem]{Conjecture}
\newtheorem{corollary}[theorem]{Corollary}
\newtheorem{lemma}[theorem]{Lemma}
\newtheorem{proposition}[theorem]{Proposition}

\theoremstyle{definition}
\newtheorem{definition}[theorem]{Definition}

\newtheorem{example}[theorem]{Example}

\newtheorem{remark}[theorem]{Remark}

\newcommand{\rr}{\mathbb{R}}

\newcommand{\calB}{\mathcal{B}}

\newcommand{\calI}{\mathcal{I}}

\newcommand{\calL}{\mathcal{L}}

\newcommand{\calN}{\mathcal{N}}

\newcommand{\calP}{\mathcal{P}}
\newcommand{\calQ}{\mathcal{Q}}
\newcommand{\calR}{\mathcal{R}}
\newcommand{\calS}{\mathcal{S}}
\newcommand{\calT}{\mathcal{T}}

\newcommand{\MR}{\mathrm{MR}}
\newcommand{\Gr}{\mathit{Gr}}

\usepackage[all]{xy}
\xyoption{poly}
\xyoption{arc}

\usetikzlibrary{decorations.pathmorphing,calc}
\usetikzlibrary{intersections}

\definecolor{light-gray}{gray}{0.6}

\tikzstyle{propagator}=[decorate,decoration={snake,amplitude=0.8mm}]
\tikzstyle{smallpropagator}=[decorate,decoration={snake,segment length=3mm,amplitude=0.5mm}]

\tikzstyle{firstdash}=[dashed,line cap=round, dash pattern=on 2pt off 1pt]
\tikzstyle{seconddash}=[dashed,line cap=round, dash pattern=on 0.5pt off 1pt]

\newcommand{\drawWLD}[2]{

\pgfmathsetmacro{\n}{#1}
\pgfmathsetmacro{\radius}{#2}
\pgfmathsetmacro{\angle}{360/\n}
    \foreach \i in {1,2,...,\n} {
      \pgfmathsetmacro{\x}{\angle*\i}
        \draw[-,shorten >=-\radius*0.1 cm,shorten <=-\radius*0.1 cm]  (\x:\radius cm)-- (\x + \angle: \radius cm);
    }

}

\newcommand{\drawprop}[4]{
\pgfmathsetmacro{\r}{#1}
\pgfmathsetmacro{\bumpr}{#2}
\pgfmathsetmacro{\s}{#3}
\pgfmathsetmacro{\bumps}{#4}
\pgfmathsetmacro{\perturbe}{\angle/\n}

\begin{scope}
\clip (\angle*\r:\radius) -- (\angle + \angle*\r:\radius) -- (\angle*\s:\radius) -- (\angle + \angle*\s:\radius) -- (\angle*\r:\radius);
\draw[propagator] (\angle*\r + \angle/2 + \bumpr*\perturbe:\radius) -- (\angle*\s + \angle/2 + \bumps*\perturbe:\radius);
\end{scope}
}

\newcommand{\drawnumbers}{
  \foreach \i in {1,2,...,\n} {
  \pgfmathsetmacro{\x}{\angle*\i}
  \draw (\x:\radius*1.15) node {\footnotesize \i};
}
}

\title{Basis shape loci and the positive Grassmannian}
\author{Cameron Marcott}
\date{}

\begin{document}
\maketitle

\begin{abstract}
\noindent
A basis shape locus takes as input data a zero/nonzero pattern in an $n \times k$ matrix, which is equivalent to a presentation of a transversal matroid. The locus is defined as the set of points in $\Gr(k,n)$ which are the row space of a matrix with the prescribed zero/nonzero pattern. We show that this locus depends only on the transversal matroid, not on the specific presentation. When a transversal matroid is a positroid, the closure of its basis shape locus is the associated positroid variety. We give a sufficient, and conjecturally necessary, condition for when a transversal matroid is a positroid. Finally, we discus applications to two programs for computing scattering amplitudes in $\calN = 4$ SYM theory: one trying to prove that projections of certain positroid cells triangulate the amplituhedron, and another using Wilson loop diagrams.
\end{abstract}

\section{Introduction}

Throughout this paper, $\calS = \{S_1, \dots, S_k\}$ will be a family of subsets of $[n]$. The family $\calS$ is unordered and may contain repeated sets. Say that $S_i$ is the {\textit{support set}} of a vector $\mathbf{v}_i$, $\mathrm{supp}(\mathbf{v}_i) = S_i$, if the $j^{th}$ coordinate of $\mathbf{v}_i$ is nonzero if and only if $j \in S_i$ for all $j \in [n]$.

\begin{definition} \label{def:basis_shape_locus}
Given a set system $\calS$, the {\textit{basis shape locus}} $L(\calS)$ associated to $\calS$ is the subset
\begin{displaymath}
\left\{
\mathrm{span}( \mathbf{v}_1, \mathbf{v}_2, \dots, \mathbf{v}_k) :
\mathbf{v}_i \in \rr^{[n]}, \mathrm{rk}(\mathbf{v}_1, \dots, \mathbf{v}_k) = k,  \mathrm{supp}(\mathbf{v}_i) = S_i \mbox{ for all $i$}
\right\}
\end{displaymath}
\noindent
of $\Gr(k,n)$.
\end{definition}

For $1 \leq i \leq k$, $1 \leq j \leq n$, let $x_{ij}$ be algebraically independent, invertible variables. Let $M_{\calS}(\mathbf{x})$ be the $k \times n$ matrix whose $i,j$ entry is $x_{ij}$ if $j \in S_i$ and is zero otherwise. So, any point in $L(\calS)$ may be obtained by evaluating the $x_{ij}$ at nonzero real numbers in $M_{\calS}(\mathbf{x})$, then taking the row span of the resulting matrix.

Also associated to $\calS$ is a transversal matroid $\calB(\calS)$, which is the matroid represented by a generic point in $L(\calS)$. The set system $\calS$ is referred to as a {\textit{presentation}} of $\calB(\calS)$. In general, transversal matroids have multiple presentations. Background on transversal matroids is provided in Section \ref{sec:transversal_matroids}. Theorem \ref{thm:locus_is_matroidal} shows that the locus $L(\calS)$ depends only on the matroid $\calB(\calS)$, not the presentation $\calS$.

A naive upper bound on the dimension of $L(\calS)$ is
\begin{equation}
\mathrm{nmd}(\calS) = -k + \sum_{i = 1}^{k} |S_i|,
\end{equation}
obtained by scaling each row of $M_{\calS}(\mathbf{x})$ so that one entry is $1$, then evaluating the other parameters freely. We call $\mathrm{nmd}(\calS)$ the {\textit{naive maximal dimension}} of $L(\calS)$. Theorem \ref{thm:dimension_and_minimal_presentations} shows that $\mathrm{dim}(L(\calS)) = \mathrm{nmd}(\calS)$ if and only if $\calS$ is a {\textit{minimal presentation}} of $\calB(\calS)$.

\begin{example} \label{ex:nmd}
Let $\calS$ be the set system consisting of
\begin{displaymath}
\begin{split}
S_1 & = \{1, 3, 4\}, \\
S_2 & = \{1, 2\}, \\
S_3 & = \{2, 3\}.
\end{split}
\end{displaymath}
\noindent
Then,
\begin{displaymath}
M_{\calS}(\mathbf{x}) =
\left(
\begin{array}{cccc}
x_{11} & 0 & x_{13} & x_{14} \\
x_{21}& x_{22} & 0 & 0 \\
0 & x_{32} & x_{33} & 0
\end{array}
\right).
\end{displaymath}
\noindent
Here, $\mathrm{nmd}(\calS) = 4$. However, we may eliminate a degree of freedom via
\begin{displaymath}
\left(
\begin{array}{ccc}
1 & \frac{-x_{11}}{x_{21}} & \frac{x_{11}x_{22}}{x_{21}x_{32}} \\
0 & 1 & 0 \\
0 & 0 & 1
\end{array}
\right)
M_{\calS}(\mathbf{x}) =
\left(
\begin{array}{cccc}
0 & 0 & x_{13} + \frac{x_{11}x_{22}x_{33}}{x_{21}x_{32}} & x_{14} \\
x_{21}& x_{22} & 0 & 0 \\
0 & x_{32} & x_{33} & 0
\end{array}
\right).
\end{displaymath}
\noindent
We cannot eliminate any more degrees of freedom from here, and the dimension of $L(\calS)$ is $3$. The support sets of this new matrix's rows give a set system $\calS'$ with $L(\calS') = L(\calS)$ and $\mathrm{nmd}(\calS') = 3$. Moreover, $\calS'$ is a minimal presentation its transversal matroid.
\end{example}

Theorem \ref{thm:loci_are_positroids} shows that if the transversal matroid $\calB(\calS)$ is a positroid, then $\overline{L(\calS)}$ is the {\textit{positroid variety}} labelled by this matroid. Positroids are introduced in Section \ref{sec:positroid}. Theorem \ref{thm:no_crossing_implies_positroid} gives a sufficient condition for testing whether the matroid $\calB(\calS)$ is a positroid. We conjecture that this condition is also necessary. Section \ref{sec:transversal_positroids} proves this conjecture in several special cases, including when $\overline{L(\calS)}$ is a Richardson variety, and when all sets in $\calS$ have the same size. Section \ref{sec:comparison_with_other_classes} compares the basis shape loci to several similarly defined families of subsets of $\Gr(k,n)$, including generalized Richardson varieties introduced by Billey and Coskun in \cite{billey:generalized_richardson_varieties}, interval positroids introduced by Knutson in \cite{knutson:interval_positroid}, and diagram varieties introduce by Liu in \cite{liu:diagram_varieties}.

The motivation for this work partially comes from two programs related to scattering amplitudes in $\calN = 4$ SYM theory. The first computes the on-shell amplitudes as volumes of an object called the amplituhedron, a certain projection of $\Gr_{\geq 0}(k,n)$. In \cite{karp:decompositions}, Karp, Williams, and Zhang give a program for triangulating the amplituhedron. Their program identifies certain positroid cells, called BCFW cells, finds a basis with a prescribed support shape for any plane in these cells, then uses these basis shapes and sign variation techniques to argue that the images of these cells are disjoint in the amplituhedron. The second program computes the total amplitude using a Wilson loop. This program identifies matrices of the form $M_{\calS}(\mathbf{x})$ for a particular class of set systems $\calS$, then associates an integral to this family of matrices. Section \ref{sec:applications_of_basis_shape_loci} discuses applications to these programs, and draw connections between the basis shapes appearing in them.

\section{Background and notation}

\subsection{Matroids} \label{sec:matroid}

\begin{definition} \label{def:matroid}
A {\textit{matroid}} is a collection of sets $\calB \subseteq \binom{[n]}{k}$ such that for each $S, T \in \calB$ and each $i \in S \setminus T$ there is some $j \in T \setminus S$ such that $S\setminus i \cup j$ and $T \setminus j \cup i$ are both in $\calB$.
\end{definition}

For any $I \in \binom{[n]}{k}$, let $\Delta_{I}$ be the $I^{th}$ Pl\"ucker coordinate on $\Gr(k,n)$. So, if a point in $\Gr(k,n)$ is represented as the row space of a $k \times n$ matrix, $\Delta_{I}$ is the determinant of the $k \times k$ submatrix whose columns are indexed by $I$. For any $V \in \Gr(k,n)$ the set
\begin{displaymath}
\left\{ I \in \binom{[n]}{k} : \Delta_{I}(V) \neq 0 \right\}
\end{displaymath}
\noindent
is a matroid, called the {\textit{matroid represented by $V$}}.

We quickly review some terminology from matroid theory. For a matroid $\calB$, the sets in $\calB$ are called {\textit{bases}} of the matroid. The set $[n]$ is called the {\textit{ground set}} of the matroid. A set $I \subseteq [n]$ is {\textit{independent}} if it is contained in some element of $\calB$. Otherwise, $I$ is {\textit{dependent}}. The {\textit{rank}} of $I$, $\mathrm{rk}(I)$, is the size of the largest independent subset of $I$. An element $x \in [n]$ is a {\textit{loop}} if $\mathrm{rk}(x) = 0$ and is a {\textit{coloop}} if $x \in B$ for all $B \in \calB$. A set $I$ is a {\textit{circuit}} if it is dependent and all of its proper subsets are independent. The set $I$ is a {\textit{cocircuit}} if it is a minimal set intersecting every basis. The set $I$ is a {\textit{flat}} if for all $x \in [n]\setminus I$, $\mathrm{rk}(I \cup x) = \mathrm{rk}(I) + 1$. A flat is a {\textit{cyclic flat}} if it is a union of circuits.

Given $I \subseteq [n]$, the {\textit{restriction}} of $\calB$ to $I$, $\calB|_I$, is the collection of maximal size $C \subseteq I$ such that $C \subseteq B$ for some $B \in \calB$. The {\textit{deletion}} of $I$, $\calB \setminus I$, is the collection of maximal size $C \subseteq [n] \setminus I$ such that $C \subseteq B$ for some $B \in \calB$. The {\textit{contraction}} of $I$, $\calB / I$, is the collection of $C \subseteq [n] \setminus I$ such that $C \cup D \in B$ for any maximal rank $D \subset I$. The result of a restriction, deletion, or contraction of a matroid is another matroid, called a {\textit{minor}} of $\calB$. The ground set $[n]$ is ordered $1 < 2 < \cdots < n$ and the ground sets of $\calB|_I$, $\calB \setminus I$, and $\calB / I$ inherit orderings from this ordering. This feature is not always present when studying general matroids, but is necessary when studying positivity phenomena.

The {\textit{dual}} of a matroid is $\calB^{\ast} = \{[n] \setminus B : B \in \calB\} \subseteq \binom{[n]}{n-k}$. The dual $\calB^{\ast}$ is a matroid, and $\calB^{\ast \ast} = \calB$. Deletion and contraction are dual operations in the sense that $(\calB \setminus I)^{\ast} = \calB^{\ast} / I$.

The {\textit{direct sum}} of matroids $\calB$ and $\calB'$ on disjoint ground sets, $\calB \oplus \calB'$, is the set of $B \cup B'$ where $B \in \calB$ and $B' \in \calB'$. A matroid is {\textit{connected}} if it is not direct sum of two nontrivial matroids. If $\calB$ is the direct sum of connected matroids, each constituent in this sum is a {\textit{connected component}} of $\calB$.

The {\textit{matroid polytope}} of $\calB$ is the convex hull of the indicator vectors of the sets in $\calB$ in $\rr^{[n]}$.

\begin{theorem}[Proposition 2.6 in \cite{feichtner:matroid_polytopes}]
If $\calB$ is connected, the matroid polytope of $\calB$ is the subset of the simplex $\left\{(x_1, x_2, \dots, x_n) \in \rr^{[n]}_{\geq 0} : \sum_{i = 1}^{n} x_i = k \right\}$ defined by the inequalities
\begin{displaymath}
\sum_{i \in F} x_i \leq \mathrm{rk}(F),
\end{displaymath}
\noindent
where $F$ is a flat such that $\calB|_{F}$ and $\calB/F$ are connected. Each such inequality defines a facet of the matroid polytope.
\end{theorem}

In light of this theorem, flats $F$ of $\calB$ such that $\calB|_{F}$ and $\calB/F$ are connected are called {\textit{flacets}}. All flacets containing more than one element are cyclic flats. To see this, note that if $|F| > 1$ and $\calB|_{F}$ is connected, then for every $f \in F$ there must be some basis $B_f$ of $\calB|_{F}$ not containing $f$. Then, $f \cup B_{f}$ yields a circuit containing $f$.

The {\textit{matroid strata}}, $V_\calB \subseteq \Gr(k,n)$, is the set of points in $\Gr(k,n)$ representing $\calB$. Matroid strata are also called thin Schubert cells or GGMS strata, after Gelfand, Goresky, MacPherson, and Serganova. All these names are misleading, as matroid strata are not in general cells and do not stratify the Grassmannian.

\subsection{Positroids} \label{sec:positroid}

\begin{definition} \label{def:positive_grassmannian}
The {\textit{positive Grassmannian}}, $\Gr_{\geq 0}(k,n)$, is the subset of $\Gr(k,n)$ where all Pl\"ucker coordinates have the same sign.
\end{definition}

\begin{definition} \label{def:positroid}
A {\textit{positroid}} is a matroid represented by some point in $\Gr_{\geq 0}(k,n)$.
\end{definition}

When $\calB$ is a positroid, $\overline{V_{\calB}}$ is a {\textit{positroid variety}}. Ardila, Rinc\'on, and Williams characterize positroids in terms of their flacets and connected components in \cite{ardila:non-crossing_partitions}.

\begin{theorem}[Proposition 5.6 in \cite{ardila:non-crossing_partitions}] \label{thm:positroid_flacets}
A connected matroid is a positroid if and only if every flacet is a cyclic interval.
\end{theorem}

\begin{theorem}[Theorem 7.6 in \cite{ardila:non-crossing_partitions}] \label{thm:positroids_noncrossing_partitions}
A matroid $\calB$ is positroid if and only if each connected component of $\calB$ is a positroid and the connected components of $\calB$ form a noncrossing partition of $[n]$.
\end{theorem}

When $\calB$ is a positroid, Marsh and Reitsch in \cite{marsh:parametrizations} define a cell $\MR(\calB) \cong (\rr^{\ast})^{\dim(V_\calB)}$ and a matrix $\MR_{\calB}(\mathbf{x})$ parameterizing this cell, which we respectively call the {\textit{Marsh-Reitsch cell}} and {\textit{Marsh-Reitsch matrix}} of $\calB$. This cell this cell contains $V_{\calB}$, and its positive part provides a parameterization of $V_\calB \cap Gr_{\geq 0}(k,n)$. Marsh and Reitsch's construction gives a decomposition of any flag variety; explicit combinatorial details for the Grassmannian are due to Postnikov in \cite{postnikov:total_positivity}. We collect a few facts about $\MR(\calB)$.

\begin{theorem} \label{thm:facts_about_marsh_reitsh_cell}
Let $\calB$ be a positroid and $\MR(\calB) \subset \Gr(k,n)$ be the associated Marsh-Reitsch cell.
\begin{itemize}
\item[(i)] $\MR(\calB) \cong (\rr^{\ast})^{\dim(V_\calB)}$, Proposition 5.2 in \cite{marsh:parametrizations}.
\item[(ii)] There is a matrix $\MR_{\calB}(\mathbf{x})$ in $\dim(V_\calB)$ indeterminants, such that evaluating $\mathbf{x} \in (\rr^{\ast})^{\dim(V_\calB)}$ provides a parameterization of $\MR(\calB)$, Proposition 5.2 in \cite{marsh:parametrizations}.
\item[(iii)] $V_{\calB} \subseteq \MR(\calB)$, Corollary 7.9 in \cite{talaska:network_parameterizations}.
\item[(iv)] $V_\calB \cap Gr_{\geq 0}(k,n) = \MR(\calB) \cap Gr_{\geq 0}(k,n) \cong \rr_{> 0}^{\dim(V_\calB)}$. This space is parameterized by evaluating $\MR_{\calB}(\mathbf{x})$ at $\mathbf{x} \in (\rr_{> 0})^{\dim(V_\calB)}$, Theorem 3.8 in \cite{postnikov:total_positivity}.
\end{itemize}
\end{theorem}

\subsection{Transversal Matroids} \label{sec:transversal_matroids}

Let $\Gamma_{\calS}$ be the bipartite graph where one part has nodes labelled $S_1, S_2, \dots, S_k$, the other part has nodes labelled $1, 2, \dots, n$, and there is an edge between $S_i$ and $j$ if and only if $j \in S_i$. Define $\calB(\calS) \subseteq \binom{[n]}{k}$ to be the collection of subsets $B$ such that there is a matching saturating $B$ in $\Gamma_{\calS}$. $\calB(\calS)$ is a matroid. Matroids obtained in this way are {\textit{transversal matroids}}. The set $\calS$ is a {\textit{presentation}} of $\calB(\calS)$. If $\calB$ is a rank $k$ transversal matroid, then $\calB = \calB(\calS)$ for some $\calS$ with $|\calS| = k$. We consider only these presentations.

In general, a transversal matroid has multiple presentations. Presentations are partially ordered by $\calS \preceq \calS'$ if there is some ordering of the sets $S_1, \dots, S_k$ and $S'_1, \dots, S'_k$ in $\calS$ and $\calS'$ such that $S_i \subseteq S'_i$ for each $i$. A transversal matroid has a unique maximal presentation, and usually has several minimal presentations. For more background on transversal matroids and their presentations, see \cite{brualdi:transversal_matroids}.

\begin{example} \label{ex:presentations_of_transversal_matroids}
Let $\calS$ be the set system from Example \ref{ex:nmd}. The bipartite graph $\Gamma_{\calS}$ is the following.
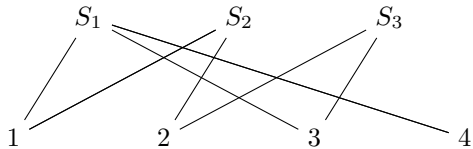
\begin{displaymath}
\begin{tikzpicture}
\draw (-2,.8) node (S1) {$S_1$};
\draw (0,.8) node (S2) {$S_2$};
\draw (2,.8) node (S3) {$S_3$};
\draw (-3,-.8) node (A) {$1$};
\draw (-1,-.8) node (B) {$2$};
\draw (1,-.8) node (C) {$3$};
\draw (3,-.8) node (D) {$4$};
\draw (S1) -- (A);
\draw (S1) -- (C);
\draw (S1) -- (D);
\draw (S2) -- (A);
\draw (S2) -- (B);
\draw (S3) -- (B);
\draw (S1) -- (D);
\draw (S2) -- (A);
\draw (S3) -- (C);
\end{tikzpicture}
\end{displaymath}

In $\Gamma_{\calS}$, there is a matching saturating every element of $\binom{[n]}{4}$. So, $\calB(\calS) = \binom{[n]}{4}$. Let $\calS' = \calS \setminus S_1 \cup \{3,4\}$. So, $\Gamma_{\calS'}$ is the graph obtained by deleting edge from $S_1$ to $1$. In $\Gamma_{\calS'}$, there is also a matching saturating every element of $\binom{[n]}{4}$. So, $\calS' \prec \calS$.
\end{example}

The following theorem appears several times in the literature, for instance following from Theorems 3.2 and 3.7 in \cite{brualdi:transversal_matroids}, or as Lemma 2 in \cite{bondy:transversal_matroids}.

\begin{theorem} \label{thm:minimal_presentation}
A presentation $\calS$ of the transversal matroid $\calB(\calS)$ is minimal if and only if each set in $\calS$ is a distinct cocircuit of $\calB(\calS)$.
\end{theorem}

\begin{theorem}[Theorem 2 in \cite{bondy:transversal_matroids}] \label{thm:min_presentations_have_same_sizes}
Let $\calS = \{S_1, \dots, S_k\}$ and $\calS' = \{S'_1, \dots, S'_k\}$ be minimal presentations of $\calB(\calS)$. The transversal matroid $\calB(\calS)$ has a unique maximal presentation $\{M_1, \dots, M_k\}$. We may reindex the sets in $\calS$ and $\calS'$ such that $S_i \cup S'_i \subseteq M_i$ for all $i$. Here, $|S_i| = |S'_i|$ for all $i$.
\end{theorem}

If $\Gamma_{\calS}$ is disconnected, then the matroid $\calB(\calS)$ is the direct sum of the transversal matroids coming from the connected components of $\Gamma_{\calS}$. If $\calS$ is a minimal presentation of $\calB(\calS)$, then $\calB(\calS)$ is a connected matroid if and only if $\Gamma_{\calS}$ is connected.

Given $\calT \subseteq \calS$, the set
\begin{equation} \label{eqn:cyclic_flat_transversal_matroid}
F(\calT) = \left\{ i \in [n] : i \notin S \quad \forall \quad S \notin \calT \right\}
\end{equation}
\noindent
is a flat of the matroid $\calB(\calS)$.

\begin{lemma}[Corollary 2.8 in \cite{bonin:introduction_to_transversal_matroids}] \label{lem:cyclic_flats_transversal_matroid}
Let $F$ be a cyclic flat of the transversal matroid $\calB(\calS)$. Then, $F = F(\calT)$ for some $\calT \subseteq \calS$ such that $|\calT| = \mathrm{rk}(F)$. In particular, if $\calB$ is connected, all flacets of $\calB$ containing more than one element are of the form $F(\calT)$.
\end{lemma}

Given $I \subseteq [n]$, $\calB(\calS)|_I$ is a transversal matroid presented by $\calS|_I = \{S \cap I : S \in \calS\}$. Even if $\calS$ is a minimal presentation of $\calB(\calS)$, $\calS|_I$ is not necessarily a minimal presentation of $\calB(\calS)|_I$. However, if $I$ is a cyclic flat of $\calB(\calS)$, then $\calS|_I$ will be a minimal presentation of $\calB(\calS)|_I$, after possibly removing copies of the empty set from $\calS|_I$.

Transversal matroids are not closed under contractions. However, for any flat $F(\calT)$ with $\mathrm{rk}(F(\calT)) = |\calT|$, $\calB(\calS)/F(\calT)$ is a transversal matroid on the ground set $[n] \setminus F(\calT)$ with presentation $\calS \setminus \calT$. If $\calS$ is a minimal presentation of $\calB(\calS)$, then $\calS \setminus \calT$ is a minimal presentation of $\calB(\calS)/F(\calT)$. In particular, this fact holds if $F(\calT)$ is a cyclic flat.

\section{Basis shape loci} \label{sec:shape_loci}

Let $\calS$ be a set system and $L(\calS)$ be its basis shape locus.

\begin{proposition} \label{prop:generic_point_represents_matroid}
A generic point in $L(\calS)$ represents the matroid $\calB(\calS)$.
\end{proposition}

\begin{proof}
Collections of vectors $\mathbf{v}_1, \mathbf{v}_2, \dots, \mathbf{v}_k \in \rr^{[n]}$ with $\mathrm{supp}(\mathbf{v}_i) = S_i$ are in bijection with edge weightings of $\Gamma_{\calS}$, with the edge between $S_i$ and $j$ weighted by the $j^{th}$ entry of $\mathbf{v}_i$. The $I^{th}$ Pl\"ucker coordinate of a point in $L(\calS)$ may be computed from $\Gamma_{\calS}$ by
\begin{displaymath}
\Delta_I = \sum_{M : \calS \ \textbf{--} \ I} \mathrm{sgn}(M)\prod_{e \in M} \mathrm{wt}(e),
\end{displaymath}
\noindent
where the sum is across matchings from $\calS$ to $I$ in $\Gamma_{\calS}$ and $\mathrm{sgn}(M)$ is the sign of $M$ viewed as a permutation. Generically, $\Delta_I$ is nonzero if and only if there is a matching in $\Gamma_{\calS}$ saturating $I$.
\end{proof}

\begin{theorem} \label{thm:dimension_and_minimal_presentations}
The following are equivalent:
\begin{itemize}
\item[(i)] $\dim(L(\calS)) = \mathrm{nmd}(\calS)$.
\item[(ii)] $\calS$ is a minimum presentation of $\calB(\calS)$.
\item[(iii)] For all $\calT \subseteq \calS$,
\begin{equation} \label{eqn:set_size_condition}
\left| \bigcup_{T \in \calT} T \right| \geq \max_{T \in \calT}\left( \left| T \right| \right) + \left| \calT \right| - 1.
\end{equation}
\end{itemize}
\end{theorem}

\begin{remark} \label{rem:ii_and_iii_equivalent}
The equivalence of (ii) and (iii) in Theorem \ref{thm:dimension_and_minimal_presentations} are proved in Theorem 3 in \cite{bondy:transversal_matroids}. There, condition (iii) is phrased as the equivalent condition that there is a matching of size $k - 1$ from $\calS \setminus T$ to $[n]\setminus T$ in $\Gamma_{\calS}$ for all $T \in \calS$. This equivalence is proved by noting that for any $\calT \subseteq \calS \setminus T$,
\begin{displaymath}
\begin{split}
\left| \bigcup_{S \in \calT} S \setminus T \right| &
= \left| \bigcup_{S \in \calT \cup T} S \setminus T \right| \\
& \geq  \max_{S \in \calT}\left( \left| S \right| \right) - |T|  + \left| \calT \cup T \right| - 1 \\
& \geq |\calT|.
\end{split}
\end{displaymath}
\noindent
These inequalities are exactly the inequalities from Hall's Matching Theorem.
\end{remark}

Taking the equivalence of (ii) and (iii) above, we establish a technical lemma which will be used to show the equivalence of (i) with the other two conditions.

\begin{lemma} \label{lem:delete_element_preserve_minimal_presentation}
Let $\calS = \{S_1, \dots, S_k\}$ be a minimal presentation of $\calB(\calS)$. If $|S| >1$ for some $S \in \calS$, there is some $S_i \in \calS$ of maximal size and some $s \in S_i$ such that $\calS' = \{S_1, \dots, S_i \setminus s, \dots, S_k\}$ is a minimal presentation of $\calB(\calS')$.
\end{lemma}

\begin{proof}
If there is a unique $S$ of maximal size in $\calS$, removing any element of $S$ preserves the inequalities (\ref{eqn:set_size_condition}). Otherwise, suppose toward contradiction that for every $S \in \calS$ of maximal size and every $s \in S$ that there is some $\calT_{S,s} \subseteq \calS \setminus S$ such that
\begin{displaymath}
\left| \bigcup_{T \in \calT_{S,s}} T \cup (S \setminus s) \right| < |\calT_{S,s}| + |S|.
\end{displaymath}

For each pair $S,s$, choose $\calT_{S,s}$ to be a maximal subset of $\calS$ with this property. The subsystems $\calT_{S,s}$ are partially ordered by inclusion. Choose $S, s$ such that $\calT_{S,s}$ is minimal in this poset. That is, $\calT_{S,s} \not\supset \calT_{S',s'}$ for any other $S' \in \calS$ of maximal size and any $s' \in S'$. Necessarily, $\calT_{S,s}$ contains some set of size $|S|$.

Since $\calT_{S,s} \cup S$ satisfies (\ref{eqn:set_size_condition}),
\begin{equation} \label{eqn:ts_is_exact}
\left| \bigcup_{T \in \calT_{S,s}} T \right| = |\calT_{S,s}| + |S| -1,
\end{equation}
\noindent
and
\begin{equation} \label{eqn:Sminuss_is_subset}
(S \setminus s) \subset \bigcup_{T \in \calT_{S,s}} T.
\end{equation}

The same holds for $\calT_{S,t}$ for any $t \neq s$. If there is some $S' \in \calT_{S,s} \cap \calT_{S,t}$ with $|S'| = |S|$,
\begin{displaymath}
\left| \bigcup_{T \in \calT_{S,s} \cup \calT_{S,t}} T \right| = | \calT_{S,s} \cup \calT_{S,t}| + |S'| - 1.
\end{displaymath}
\noindent
Since $S \subset \bigcup_{T \in \calT_{S,s} \cup \calT_{S,t}} T$, $\calT_{S,s} \cup \calT_{S,t} \cup S$ violates the inequality (\ref{eqn:set_size_condition}), contradicting the assumption that $\calS$ was a minimal presentation. Thus,
\begin{displaymath}
\left\{ S' \in \calT_{S,s} \cap\calT_{S,t} : |S'| = |S| \right\} = \emptyset.
\end{displaymath}

Let $S' \in \calT_{S,s}$ be of maximal size and $s' \in S'$. Since $\calT_{S',s'}$ and $\calT_{S',t'}$ likewise cannot maximal size share sets for $s' \neq t'$, we may assume that $S \notin \calT_{S',s'}$. Since $\calT_{S,s}$ was chosen to be minimal in the poset on set systems ordered by containment, $\calT_{S',s'} \not\subset \calT_{S,s}$.

Let $\calT = \calT_{S,s} \cup \calT_{S',s'}$. Applying (\ref{eqn:ts_is_exact}) and (\ref{eqn:Sminuss_is_subset}) to $\calT_{S',s'}$, $\left|\bigcup_{T \in \calT_{S',s'} \cup S'}\right| = |\calT_{S',s'}| + |S'|$. Since $S' \in (\calT_{S',s'} \cup S') \cap \calT_{S,s}$,
\begin{displaymath}
\left| \bigcup_{T \in \calT_{S,s} \cup \calT_{S',s'}} T  \right| = |\calT_{S,s} \cup \calT_{S',s'}| + |S'|-1.
\end{displaymath}

By the maximality of $\calT_{S,s}$, $s \in \bigcup_{T \in \calT_{S,s} \cup \calT_{S',s'}} T$. So,
\begin{displaymath}
\left| \bigcup_{T \in \calT_{S,s} \cup \calT_{S',s'}} T  \cup S \right| = |\calT_{S,s} \cup \calT_{S',s'}| + |S'|-1,
\end{displaymath}
violating the inequality (\ref{eqn:set_size_condition}).
\end{proof}

\begin{example}
This example illustrates that while one may always remove some element from some set in a minimal presentation of transversal matroid to achieve a minimal presentation of a different transversal matroid, not every maximal size set contains an element which may be removed to produce a minimal presentation. Consider the set system $\{12, 23, 34\}$, which is a minimal presentation of its transversal matroid. Removing $2$ from the first set produces the set system $\{1, 23, 34\}$, which satisfies the inequalities (\ref{eqn:set_size_condition}) and is thus a minimal presentation of its transversal matroid. Likewise, removing $3$ from the third set gives a set system satisfying (\ref{eqn:set_size_condition}). Any other choice of element to remove produces a set system which is not a minimal presentation.
\end{example}

\begin{proof}[Proof of Theorem \ref{thm:dimension_and_minimal_presentations}]
The equivalence of (ii) and (iii) are shown in Theorem 3 in \cite{bondy:transversal_matroids}. We begin by showing that point (i) implies point (iii).

Let $M_{\calS}(\mathbf{x})$ be the matrix whose entries algebraically independent invertible variables $x_{ij}$ or zeros. Let $\mathbf{v}_i$ be the $i^{th}$ row of $M_\calS(\mathbf{x})$. Suppose one of the inequalities (\ref{eqn:set_size_condition}) is violated. So, there is some set $T \subseteq \{1,\dots, k\}$ such that
\begin{displaymath}
\left| \bigcup_{j \in T} S_j \right| < \max_{j \in T}(|S_j|) + |T| - 1.
\end{displaymath}
\noindent
Take $T$ to be such that no proper subset of $T$ has this property.
Let $S_{i_1}$ be a maximal size set in $\calT = \{S_j : j\in T\}$ and let $a_1 \in S_{i_1} \cap \left(\cup_{j \in T \setminus i_1} S_j\right)$. The minimality of $T$ implies $\Gamma_{\calT}$ is connected, and thus such an $a_1$ exists. Let $i_2 \in T \setminus i_i$ be such that $a_1 \in S_{i_2}$. Then,
\begin{displaymath}
\mathrm{supp}\left( \mathbf{v}_{i_1} - \frac{x_{i_1 a_1}}{x_{i_2 a_1}} \mathbf{v}_{i_2}\right) = (S_{i_1} \cup S_{i_2}) \setminus a_1.
\end{displaymath}

Again from the minimality of $T$, these is some $a_2 \in \left(S_{i_1} \cup S_{i_2}\right) \cap (\cup_{j \in T \setminus i_1,i_2} S_j )$. Say, $i_3 \in T \setminus i_1, i_2$ such that $a_2 \in S_{i_3}$. Then, there is some $c, d \in \rr({\mathbf{x}})$ such that 
\begin{displaymath}
\mathrm{supp}\left( \mathbf{v}_{i_1} - \frac{x_{i_1 a_1}}{x_{i_2 a_1}} \mathbf{v}_{i_2} - c \left(\mathbf{v}_{i_3}-d\mathbf{v}_{i_2}\right) \right)
= (S_{i_1} \cup S_{i_2} \cup S_{i_3}) \setminus a_1, a_2.
\end{displaymath}

Continuing to eliminate variables in this way, we may replace $\mathbf{v}_{i_1}$ with some vector $\mathbf{v}'_{i_1}$ supported on $S'_1 = \bigcup_{j \in T} S_j \setminus \{a_1, \dots, a_{|T|-1}\}$ without altering $L(\calS)$. Then,
\begin{displaymath}
\left|S'_1\right| = \left| \bigcup_{j \in T} S_j \right| - |T| + 1 < \left|S_1 \right|.
\end{displaymath}
\noindent
Let $\calS' = \calS \setminus S_1 \cup S'_1$. Since $L(\calS) = L(\calS')$,
\begin{displaymath}
\dim(L(\calS)) = \dim(L(\calS')) \leq \mathrm{nmd}(\calS') < \mathrm{nmd}(\calS),
\end{displaymath}
\noindent
and thus point (i) implies point (iii).

We establish that point (ii) implies point (i) by inducting on $\dim(L(\calS))$. If $\dim(L(\calS)) = 0$, exactly one Pl\"ucker coordinate $\Delta_I$ is nonzero on $L(\calS)$. Say $I = \{i_1, \dots, i_k\}$. $\calS = \{ \{i_1\}, \dots, \{i_k\}\}$ is the unique minimal presentation of $\calB(\calS)$, and $\mathrm{nmd}(\calS) = 0$ as well.

Let $\calS$ be a minimal presentation of $\calB(\calS)$ and $\dim(L(\calS)) > 0$. Lemma \ref{lem:delete_element_preserve_minimal_presentation} guarantees some $S_i \in \calS$ and $s \in S_i$ so that $\calS' = \calS \setminus S_i \cup (S_i \setminus s)$ is a minimal presentation of $\calB(\calS')$.

We claim that $\dim(L(\calS)) > \dim(L(\calS'))$. Note that $\calB(\calS) \supset \calB(\calS')$. So, some Pl\"ucker coordinate $\Delta_I$ vanishes on $L(\calS')$, but is generically nonzero on $L(\calS)$. Pick a generic $V \in L(\calS')$ represented by $M_{\calS'}(\mathbf{y})$ for some $y_{ij} \in \rr^{\ast}$. Perturbing $\mathbf{y}$ gives a $\dim(L(\calS'))$-dimensional neighborhood around $V$. Both $V$ and this neighborhood are in $\overline{L(\calS)}$, since they are limits as the $i,s$ entry of $M_{\calS}(\mathbf{x})$ goes to zero. In $\overline{L(\calS)}$, we may perturb the $i,s$ entry of $M_{\calS}(\mathbf{y})$. Points obtained in this way are not in $V$'s neighborhood in $L(\calS')$, since $\Delta_I \neq 0$ at these points. So, $\dim(\overline{L(\calS)}) > \dim(\overline{L(\calS')})$.

Finally, inducting on $\dim(L(\calS))$,
\begin{displaymath}
\mathrm{nmd}(\calS) \geq \dim(L(\calS)) > \dim(L(\calS')) = \mathrm{nmd}(\calS') = \mathrm{nmd}(\calS) - 1.
\end{displaymath}
Thus, $\mathrm{nmd}(\calS) = \dim(L(\calS))$, completing the proof.
\end{proof}

The geometric content of this theorem is perhaps somewhat surprising. The locus $L(\calS)$ is the image of $\{M_\calS(\mathbf{y}) : y_{ij} \in \rr^{\ast}\}$ under quotient by left action of $\mathit{Gl}(k)$. Fixing one of the $y_{ij}$ to a generic parameter, one does not expect the dimension of this quotient to change. However, setting a $y_{ij}$ to zero, one generally expects the dimension of this quotient to drop by one. Contrary to this expectation, the equivalence of (i) and (iii) together with the elimination argument used to prove that (i) implies (iii) show that if the set of matrices $\{M_\calS(\mathbf{y}) : y_{ij} \in \rr^{\ast}, \mathrm{rk}(M_{\calS}(\mathbf{y})) = k\}$ and the quotient $L(\calS)$ differ by more than a torus worth of symmetry, we may always set one of the $y_{ij}$ to zero without altering the quotient $L(\calS)$. As a corollary, we see that if $\calS$ is not a minimal presentation of $\calB(\calS)$, there is some $S_i \in \calS$ and some $s \in S_i$ such that $\calB(\{S_1, \dots, S_i \setminus s, \dots, S_k\}) = \calB(\calS)$. This combinatorial fact has been noted several times in the literature, for instance as Theorem 3.7 in \cite{brualdi:transversal_matroids}, but does not a priori imply its geometric analog.

\begin{theorem} \label{thm:locus_is_matroidal}
Let $\calS$ and $\calS'$ be set systems. Then, $\calB(\calS) = \calB(\calS')$ if and only if $L(\calS) = L(\calS')$.
\end{theorem}

\begin{proof}
If $L(\calS) = L(\calS')$, Proposition \ref{prop:generic_point_represents_matroid} implies that $\calB(\calS) = \calB(\calS')$. Conversely, let $\calS$ and $\calS'$ be set systems such that $\calB(\calS) = \calB(\calS')$. It suffices to show that $L(\calS) = L(\calS')$ in the case where $\calS'$ is the unique maximal presentation of $\calB(\calS)$ guaranteed by Theorem \ref{thm:min_presentations_have_same_sizes}. The variable elimination argument used to show show point (i) implies point (iii) in Theorem \ref{thm:dimension_and_minimal_presentations} shows that we can set variables in $M_{\calS'}(\mathbf{x})$ to zero without altering the set $L(\calS')$ until we arrive at the matrix $M_{\calS}(\mathbf{x})$. Hence, $L(\calS) = L(\calS')$.
\end{proof}

\section{Dimension computations} \label{sec:dimension_computations}

This section is dedicated to proving the following theorem, which will be used later to show that $\overline{L(\calS)} = \overline{V_{\calB(\calS)}}$ in the case where $\calB(\calS)$ is a positroid.

\begin{theorem} \label{thm:positroid_dimension}
If $\calS$ is a set system such that $\calB(\calS)$ is a positroid, then
\begin{displaymath}
\dim(L(\calS)) = \dim(V_{\calB(\calS)}).
\end{displaymath}
\end{theorem}

Evidently, $\dim(L(\calS)) \leq \dim(V_{\calB(\calS)})$. In \cite{ford:expected_codimension}, Ford introduces a notion of expected codimension of a matroid variety, and proves that positroid varieties achieve their expected codimension. We prove that for a transversal matroid, Ford's expected codimension agrees with $\mathrm{codim}(L(\calS))$. For full details on the expected codimension, see \cite{ford:expected_codimension}.

\begin{definition} \label{def:expected_codimension}
Let $\calB$ be a rank $k$ matroid on the ground set $[n]$ and let $\calI$ be some collection of subsets of $[n]$. For $I \in \calI$, let
\begin{displaymath}
c(I) = |I| - \mathrm{rk}(I),
\end{displaymath}
\noindent
and
\begin{equation} \label{eqn:definition_of_b}
b_{\calI}(I) = \sum_{J \in \calI} (k - \mathrm{rk}(J)) \mu_{\calI}(I,J),
\end{equation}
\noindent
where $\mu_{\calI}$ is the M\"obius function on the poset obtained by ordering the elements of $\calI$ by containment. Then, the {\textit{expected codimension of $\calB$ with respect to}} $\calI$ is
\begin{displaymath}
\mathrm{ec}_{\calI}(\calB) = \sum_{I \in \calI} c(I) b_{\calI}(I).
\end{displaymath}
The {\textit{expected codimension of}} $\calB$ is
\begin{displaymath}
\mathrm{ec}(\calB) = \mathrm{ec}_{\calP([n])}(\calB),
\end{displaymath}
\noindent
where $\calP([n])$ is the power set of $[n]$.
\end{definition}

\begin{theorem}[Theorem 3.6 in \cite{ford:expected_codimension}] \label{thm:expected_comdimension_via_flacets}
Let $\calB$ be a connected matroid and suppose $\calI$ is a collection of subsets of $[n]$ containing every set $I$ such that $\calB|_{I}$ and $\calB/I$ are connected (i.e$.$ the flacets of $\calB$) and such that whenever $I \in \calI$, all $J$ such that $\calB|_{J}$ is a connected component of $\calB|_I$ are in $\calI$ as well. Then, $\mathrm{ec}_{\calI}(\calB) = \mathrm{ec}(\calB)$.
\end{theorem}

\begin{lemma}[Proposition 3.7 in \cite{ford:expected_codimension}] \label{lem:expected_codimension_direct_sum}
Let $\calB_1, \calB_2$ be matroids. Then,
\begin{displaymath}
\mathrm{ec}(\calB_1 \oplus \calB_2) = \mathrm{ec}(\calB_1) + \mathrm{ec}(\calB_2).
\end{displaymath}
\end{lemma}

\begin{theorem}[Theorem 4.7 in \cite{ford:expected_codimension}] \label{thm:positroid_expected_codimension}
If $\calB$ is a positroid, $\mathrm{ec}(\calB) = \mathrm{codim}(\overline{V_{\calB}})$.
\end{theorem}

\begin{theorem} \label{thm:expected_codimension_of_transversal_matroid}
Let $\calS$ be a presentation of the transversal matroid $\calB(\calS)$. Then,
\begin{displaymath}
\mathrm{codim}(\calL(\calS)) = \mathrm{ec}(\calB(\calS)).
\end{displaymath}
\end{theorem}

\begin{proof}
Since $\dim(L(\calS))$ and $\mathrm{ec}(\calB(\calS))$ both behave under direct sums of matroids, we may suppose $\calB(\calS)$ is connected. Since $\dim(\calL(\calS))$ is invariant under the presentation of the matroid, we may suppose $\calS$ is a minimal presentation of $\calB(\calS)$.

Let $\calI$ be the collection of all individual elements of $[n]$ together with all flats $F(\calT)$ such that $\mathrm{rk}(F(\calT)) = |\calT|$. From Lemma \ref{lem:cyclic_flats_transversal_matroid}, $\calI$ contains all flacets of $\calB(\calS)$. If $\mathrm{rk}(F(\calT)) = |\calT|$, the connected components of $F(\calT)$ are either coloops or sets of the form $F(\calT')$ for some $\calT' \subseteq \calT$ with $\mathrm{rk}(F(\calT')) = |\calT'|$. So, $\calI$ satisfies the hypotheses of Theorem \ref{thm:expected_comdimension_via_flacets}.

We claim that for any $F(\calT) \in \calI$,
\begin{equation} \label{eqn:calculation_of_b}
b_{\calI}(F(\calT)) =
\begin{cases}
1 & \mbox{if $|\calT| = k-1$,} \\
0 & \mbox{otherwise.}
\end{cases}
\end{equation}

Evidently, $b_{\calI}(F(\calS)) = 0$. Since $\calS$ is a minimal presentation, point (iii) of Theorem \ref{thm:dimension_and_minimal_presentations} implies there is matching in $\Gamma_{\calS}$ saturating $[n] \setminus S_i$ for all $S_i \in \calS$. So, $F(\calS \setminus S_i) \in \calI$ for all $S_i \in \calS$. So, $b_{\calI}(F(\calS\setminus S_i)) = 1$ for all $S_i \in \calS$.

Let $F(\calT) \in \calI$ such that $|\calT|<k-1$ and suppose (\ref{eqn:calculation_of_b}) holds for all $F(\calT')$ with $|\calT'| > |\calT|$. Applying M\"obius inversion to (\ref{eqn:definition_of_b}),
\begin{displaymath}
k - \mathrm{rk}(F(\calT)) = \sum_{F(\calT') \supseteq F(\calT)} b_{\calI}(F(\calT')).
\end{displaymath}
Since $\calT \subset \calS \setminus S_i$ for each $S_i \in \calS \setminus \calT$, and $\calS \setminus S_i \in \calI$ for each $S_i$, our inductive hypothesis reduces the sum on the right hand side to
\begin{displaymath}
k - |\calT| + b_{\calI}(F(\calT)).
\end{displaymath}
Since $F(\calT) \in \calI$, $\mathrm{rk}(F(\calT)) = |\calT|$, and so $b_{\calI}(F(\calT)) = 0$. Inductively, (\ref{eqn:calculation_of_b}) holds for all $F(\calT) \in \calI$. Since $\calB(\calS)$ is connected, for all $j \in [n]$, $\mathrm{rk}(j) = 1$ and thus $c(j) = 0$. Then,
\begin{displaymath}
\begin{split}
\mathrm{ec}(\calB(\calS)) & = \mathrm{ec}_\calI(\calB(\calS)) \\
& = \sum_{S_i \in \calS} c(F(\calS \setminus S_i)) b_{\calI}(F(\calS \setminus S_i)) + \sum_{j \in [n]} c(j) b_{\calI}(j) \\
& = \sum_{S_i \in \calS} (n - |S_i|) - k(k-1) .
\end{split}
\end{displaymath}

From Theorem \ref{thm:dimension_and_minimal_presentations}, $\dim(\calL(\calS)) = \sum_{S_i \in \calS} |S_i| -k$. Since $\dim(\Gr(k,n)) = k(n-k)$,
\begin{displaymath}
\begin{split}
\mathrm{codim}(\calL(\calS)) & = k(n-k) - \left(\sum_{S_i \in \calS} |S_i| - k\right) \\
& = \mathrm{ec}(\calB(\calS)). \qedhere
\end{split}
\end{displaymath}
\end{proof}

\begin{proof}[Proof of Theorem \ref{thm:positroid_dimension}]
Let $\calS$ be a set system such that $\calB(\calS)$ is a positroid. Combining Theorems \ref{thm:positroid_expected_codimension} and \ref{thm:expected_codimension_of_transversal_matroid},
\begin{displaymath}
\dim(V_{\calB(\calS)}) = k(n-k) - \mathrm{ec}(\calB(\calS)) = \dim(\calL(\calS)).
\qedhere
\end{displaymath}
\end{proof}

\section{Basis shape loci and positroids} \label{sec:basis_shape_loci_and_positroids}

\begin{theorem} \label{thm:loci_are_positroids}
Let $\calS$ be a set system. If $\calB(\calS)$ is a positroid, then $\overline{L(\calS)} = \overline{V_{\calB(\calS)}}$, the positroid variety labelled by $\calB(\calS)$.
\end{theorem}

\begin{proof}
Suppose that $\calB(\calS)$ is a positroid. By Proposition \ref{prop:generic_point_represents_matroid}, $V_{\calB} \cap L(\calS)$ is dense in $L(\calS)$. So, $\overline{L(\calS)} \subseteq \overline{V_{\calB(\calS)}}$.

Using Theorem \ref{thm:locus_is_matroidal}, we may suppose that $\calS$ is a minimal presentation of $\calB(\calS)$. Let $V \in L(\calS)$ be the row space of a matrix  $M_{\calS}(\mathbf{y})$ obtained by evaluating the $x_{ij}$ at algebraically independent real numbers $y_{ij}$ in $M_{\calS}(\mathbf{x})$. Then, $V$ represents the matroid $\calB(\calS)$, and thus $V$ is in the Marsh-Reitsch cell $\MR(\calB(\calS))$. Let $\MR_{\calB(\calS)}(V)$ be the Marsh-Reitsch matrix representing $V$. Then, there is some $G \in Gl(k)$ such that
\begin{displaymath}
G \cdot \MR_{\calB(\calS)}(V) = M_{\calS}(\mathbf{y}).
\end{displaymath}

The Marsh-Rietsch matrix has exactly $\dim(V_{\calB})$ many free parameters. From Theorem \ref{thm:positroid_dimension}, $\dim(L(\calS)) = \dim(V_{\calB(\calS)})$ as well. So, $\MR_{\calB(\calS)}(V)$ must have been obtained by evaluating the entries of the Marsh-Rietsch matrix at algebraically independent parameters. Then, the change of basis matrix $G$ provides a basis of shape $\calS$ for a generic point in $\MR(\calB(\calS))$. So, $L(\calS) \cap \MR(\calB(\calS))$ is dense in $\MR(\calB(\calS))$ and thus $\overline{V_{\calB(\calS)}} = \overline{\MR(\calB(\calS))} \subseteq \overline{L(\calS)}$.
\end{proof}

\begin{remark}
In the case where $\calB(\calS)$ is not a positroid, little is known about $\overline{L(\calS)}$.
\end{remark}

\section{When is a transversal matroid a positroid?} \label{sec:transversal_positroids}

This section addresses the question of characterizing when a transversal matroid is a positroid. We call transversal matroids which are also positroids {\textit{transversal positroids}}.

\begin{example}
Let $\calS = \{134, 24\}$. Any point in $L(\calS)$ is the row span of a unique matrix of the form
\begin{displaymath}
\left(
\begin{array}{cccc}
1 & 0 & y_{13} & y_{14} \\
0 & 1 & 0 & y_{24}
\end{array}
\right),
\end{displaymath}
where $y_{ij} \in \rr^{\ast}$. Consider $L(\calS) \cap \Gr_{\geq 0}(k,n)$. For any point in this intersection, since $\Delta_{14} > 0$, we must have $y_{24} > 0$. Since $\Delta_{34} > 0$, $y_{13} > 0$ as well. However, this restriction forces $\Delta_{23} < 0$. Thus, $L(\calS) \cap \Gr_{\geq 0}(k,n) = \emptyset$. If $\calB(\calS)$ is a positroid $\overline{L(\calS)} = \overline{V_{\calB(\calS)}}$. However, when $\calB$ is a positroid, $\dim(V_{\calB} \cap \Gr_{\geq 0}(k,n)) = \dim(V_{\calB})$, a contradiction.

The fundamental obstruction illustrated in this example is that the set $S_1$ {\textit{crosses}} $S_2$, in the sense of Definition \ref{def:crossing_sets} below.
\end{example}

For any $a \in [n]$, let $\leq_a$ denote that $a^{th}$ cyclic shift of the usual total order on $[n]$. So, $a <_{a} a+1 <_{a} \cdots <_{a} n <_{a} 1 <_{a} \cdots <_{a} a-1$.

\begin{definition} \label{def:crossing_sets}
In a set system $\calS$, $S_i$ {\textit{crosses}} $S_j$ if there are $a, b, c, d \in [n]$ such that:
\begin{itemize}
\item[(i)] $a <_{a} b <_{a} c <_{a} d$,
\item[(ii)] $a, c \in S_i$, $a, c \notin S_j$,
\item[(iii)] $b, d \in S_j$, $b \notin S_i$.
\end{itemize}
\noindent
The set system $\calS$ is {\textit{noncrossing}} if there is no pair $S_i, S_j \in \calS$ with $S_i$ crossing $S_j$.
\end{definition}

This definition is not symmetric; $S_i$ can cross $S_j$ without $S_j$ crossing $S_i$.

\begin{theorem} \label{thm:no_crossing_implies_positroid}
Suppose that $\calS$ is a minimal presentation of $\calB(\calS)$ and that no set crosses another in $\calS$. Then, $\calB(\calS)$ is a positroid.
\end{theorem}

\begin{proof}
Suppose that $\calB(\calS)$ is not a positroid. By Theorem \ref{thm:positroids_noncrossing_partitions}, $\calB(\calS)$ can fail to be a positroid by having a connected component which is not a positroid, or by having its connected components form a crossing partition. Since $\calS$ is a minimal presentation, connected components of $\calB(\calS)$ correspond to connected components of $\Gamma_{\calS}$. Evidently, crossing connected components in $\Gamma_{\calS}$ necessitate crossing sets in $\calS$. So, we may suppose $\calB(\calS)$ is connected.

Theorem \ref{thm:positroid_flacets} says there is a flacet $F$ of $\calB(\calS)$ which is not a cyclic interval. Lemma \ref{lem:cyclic_flats_transversal_matroid} implies $F = F(\calT)$ for some $\calT \subset \calS$ with $\mathrm{rk}(\calT) = |\calT|$. Since $F$ is a cyclic flat, $\calT|_F$ is a minimal presentation of $\calB(\calS)|_{F}$. Then, $\Gamma_{\calT|_{F}}$ is connected since $\calB(\calS)|_{F}$ is connected. So, there is an $S_i \in \calT$ and $a, c \in S_i$ where $a$ and $c$ are in different cyclic intervals of $F$.

Since $\mathrm{rk}(F) = |\calT|$, $\calB(\calS)/F$ is a transversal matroid with presentation $\calS \setminus \calT$ and this presentation is minimal. Then, there os some $S_j \in \calS \setminus \calT$ amd $b, d \in S_j$ where $b$ and $d$ are in different cyclic intervals of $[n] \setminus F$.

By the definition of $F(\calT)$, $a, c \notin S_j$. Since $\calS$ is a minimal presentation, Theorem \ref{thm:dimension_and_minimal_presentations} implies $S_j \nsubseteq S_i$. So, we may take at least one of $b$ or $d \notin S_i$. Then, $S_i$ crosses $S_j$.
\end{proof}

\begin{conjecture} \label{conj:transversal_positroids}
Let $\calB$ be a transversal matroid which is also a positroid. Then, there is a noncrossing minimal presentation $\calS$ of $\calB$.
\end{conjecture}

We will show that Conjecture holds in the case where $\overline{V_{\calB}}$ is a Richardson variety, and the case where all sets in a minimal presentation of the matroid $\calB$ have the same size. A strengthening of this conjecture, providing an algorithmic method of producing a noncrossing minimal presentation is given in Conjecture \ref{conj:refined}.

Let $I = \{i_1, i_2, \dots, i_k\}$ and $J = \{j_1, j_2, \dots, j_k\}$ be subsets of $[n]$ with $i_1 < i_2 < \cdots < i_k$ and $j_1 < j_2 < \cdots < j_k$. Say that $I \leq J$ if $i_\ell \leq j_\ell$ for all $1 \leq \ell \leq k$. This partial order is called the {\textit{Gale order}} on $\binom{[n]}{k}$. If $I \leq J$, the {\textit{lattice path matroid}} defined by $I, J$ is
\begin{equation} \label{eqn:lattice_path_matroid}
\calB(I,J) = \left\{B \in \binom{[n]}{k} : I \leq B \leq J\right\}.
\end{equation}

The {\textit{Richardson variety}} associated to $I, J$ is $\overline{V_{\calB(I,J)}}$. All lattice path matroids are positroids and thus all Richardson varieties are positroid varieties.

\begin{proposition} \label{prop:richardsons_behave}
Let $I, J \in \binom{[n]}{k}$ with $I \leq J$. The lattice path matroid $\calB(I,J)$ is a transversal matroid with a noncrossing minimal presentation.
\end{proposition}

\begin{proof}
Let $I = \{i_1, i_2, \dots, i_k\}$ and $J = \{j_1, j_2, \dots, j_k\}$ with $i_1 < i_2 < \cdots < i_k$ and $j_1 < j_2 < \cdots < j_k$. For $1 \leq \ell \leq k$, let
\begin{displaymath}
S_\ell = \{m : i_\ell \leq m \leq j_\ell\}.
\end{displaymath}
\noindent
It is well known, e$.$g$.$ Section 4 of \cite{bonin:introduction_to_transversal_matroids}, that $\calB(I,J)$ is a transversal matroid with minimal presentation $\calS = \{S_1, S_2, \dots, S_k\}$. Since every $S_i$ is an interval, no two sets cross.
\end{proof}

\begin{proposition} \label{prop:all_sets_same_size}
Let $\calB$ be a transversal positroid such that all sets in a minimal presentation of $\calB$ have the same size. Then, $\calB$ has a noncrossing minimal presentation.
\end{proposition}

We require some technical machinery to prove this proposition. This machinery is presented in a way that is functional even when not all sets in $\calS$ have the same size, allowing us to give a refined version of Conjecture \ref{conj:transversal_positroids}. Theorems \ref{thm:pivot} and \ref{thm:matroidal_pivot} provide a procedure of pivoting between different minimal presentations of $\calB(\calS)$. Each pivot changes a single set in $\calS$ while preserving the matroid $\calB(\calS)$. Lemma \ref{lem:finding_exact_subset} shows that if $\calS$ is crossing and $\calB(\calS)$ is a positroid, that there is a pivot that removes this crossing. In the case where all sets in $\calS$ have the same size, these pivots can simultaneously remove all crossings in the set system.

\begin{definition} \label{def:exact_subset}
Call a subsystem $\calT \subseteq \calS$ {\textit{exact}} if
\begin{equation} \label{eqn:exact_subset}
\left| \bigcup_{T \in \calT} T \right| = |\calT| + \max_{T \in \calT}\left( |T| \right) - 1.
\end{equation}
\noindent
That is, the inequality (\ref{eqn:set_size_condition}) from point (iii) of Theorem \ref{thm:dimension_and_minimal_presentations} holds with equality for $\calT$.
\end{definition}

\begin{lemma} \label{lem:contracting_to_exact_subdiagram}
Let $\calS$ be a minimal presentation of $\calB(\calS)$, and let $\calT$ be an exact subsystem of $\calS$. Then, $\mathrm{rk}(F(\calS \setminus \calT)) = |\calS \setminus \calT|$.
\end{lemma}

\begin{proof}
Let $F = F(\calS \setminus \calT)$, the flat from (\ref{eqn:cyclic_flat_transversal_matroid}). For any $\calS' \subset \calS \setminus \calT$,
\begin{displaymath}
\begin{split}
\left| \bigcup_{S \in \calS'} S \cap F \right| & = \left| \bigcup_{S \in \calS' \cup \calT} S \right| - \left| \bigcup_{S \in \calT} S \right| \\
& \geq \left(|\calS' \cup \calT| + \max_{S \in \calS' \cup \calT}(|S|) - 1\right) - \left(|\calT| + \max_{S \in \calT}(|S|) - 1\right) \\
& \geq |\calS'|.
\end{split}
\end{displaymath}
\noindent
The positive part of the inequality on the second line comes from (\ref{eqn:set_size_condition}), and the negative part comes from the assumption $\calT$ was exact. Then, Hall's Matching Theorem implies there is a matching in $\Gamma_{\calS}$ from $\calS \setminus \calT$ to $F$ saturating $\calS \setminus \calT$. So, $\mathrm{rk}(F) = |\calS \setminus \calT|$.
\end{proof}

\begin{lemma} \label{lem:unions_of_exact_subsystems}
Let $\calT$ and $\calT'$ be exact subsystems with $S \in \calT, \calT'$ such that $|S| = \max_{T \in \calT \cup \calT'}(|T|)$. Then, $\calT \cup \calT'$ is an exact subsystem.
\end{lemma}

\begin{proof}
Note that
\begin{displaymath}
\begin{split}
\left| \bigcup_{T \in \calT \cup \calT'} T \right| & 
= \left| \bigcup_{T \in \calT} T \right| + \left| \bigcup_{T \in \calT'} T \right| - \left| \bigcup_{T \in \calT \cap \calT'} T \right| \\
& = |\calT| + |\calT'| + 2|S| - 2 - \left| \bigcup_{T \in \calT \cap \calT'} T \right| \\
& \leq |\calT| + |\calT'| + 2|S| - 2 - (|\calT \cap \calT'| + |S| - 1) \\
& = |\calT \cup \calT'| + |S| - 1.
\end{split}
\end{displaymath}
\noindent
The equality on the second line comes from the assumption that $\calT$ and $\calT'$ are exact. The inequality on the third line comes from (\ref{eqn:set_size_condition}). Then, (\ref{eqn:set_size_condition}) implies that $\calT \cup \calT'$ is exact.
\end{proof}

\begin{definition} \label{def:pivot}
Say $\calS$ and $\calS'$ are related by a {\textit{pivot}} if they satisfy (\ref{eqn:set_size_condition}), and
\begin{displaymath}
\calS' = \calS \setminus S \cup \{S \setminus a \cup b\},
\end{displaymath}
\noindent
where there is some exact $\calT \subseteq \calS$ containing $S$ such that $|S| = \max_{T \in \calT} \left( |T| \right)$,
\begin{displaymath}
a \in S \cap T, 
\quad
\mathrm{and}
\quad
b \in T \setminus S
\end{displaymath}
\noindent
for some $T \in \calT$.
\end{definition}

\begin{theorem} \label{thm:pivot}
Let $\calS$ and $\calS'$ be set systems satisfying the conditions of Theorem \ref{thm:dimension_and_minimal_presentations}. If $\calS$ and $\calS'$ are related by a pivot, then $L(\calS) = L(\calS')$ and thus $\calB(\calS) = \calB(\calS')$
\end{theorem}

\begin{proof}
Let $\mathbf{v}_i$ be the $i^{th}$ row vector of $M_{\calS}(\mathbf{x})$. If $\calS$ and $\calS'$ are related by a pivot replacing $S$ by $S \setminus a \cup b$, an argument identical to the proof of Theorem \ref{thm:dimension_and_minimal_presentations} showing (iii) implies (i) shows there is some linear combination
\begin{displaymath}
\mathbf{v}'_i = \mathbf{v}_i + \sum_{j \in T \setminus i} c_j \mathbf{v}_j
\end{displaymath}
\noindent
such that $\mathrm{supp}(\mathbf{v}'_i) = S_i \setminus a \cup b$. The entries of the vectors $\{\mathbf{v}_1, \dots, \mathbf{v}_k\} \setminus \mathbf{v}_i \cup \mathbf{v}'_i$ are all algebraically independent. So,
\begin{displaymath}
\mathrm{span}\left(\{\mathbf{v}_1, \dots, \mathbf{v}_k\} \setminus \mathbf{v}_i \cup \mathbf{v}'_i\right) = L(\calS'),
\end{displaymath}
\noindent
and thus $L(\calS) = L(\calS')$. Then, Proposition \ref{prop:generic_point_represents_matroid} implies that $\calB(\calS) = \calB(\calS')$.
\end{proof}

\begin{theorem} \label{thm:matroidal_pivot}
Let $\calS$ be an exact system and let $S \in \calS$ be a set of maximal size. Then,
\begin{displaymath}
\left\{S' \in \binom{[n]}{|S|} \ : \ \calB(\calS \setminus S \cup S') = \calB(\calS)\right\} = \calB^{\ast}(\calS \setminus S).
\end{displaymath}
\noindent
Using the exact system $\calS$, $S$ may be pivoted to any element in this set.
\end{theorem}

\begin{proof}
Let $S' \in \binom{[n]}{|S|}$ and suppose that $S' \notin \calB^{\ast}(\calS \setminus S)$. Since, $[n] \setminus S'$ is not a basis of $\calB(\calS \setminus S)$, $\calS \setminus S \cup S'$ violates the conditions of Theorem \ref{thm:dimension_and_minimal_presentations}. Then, Theorem \ref{thm:min_presentations_have_same_sizes} implies that $\calB(\calS \setminus S \cup S') \neq \calB(\calS)$.

Now, let $S' \in \calB^{\ast}(\calS \setminus S)$. Let $T \in \binom{[n]}{|S|}$ be some set obtained from $S$ by a series of pivots. Among all such sets, suppose $T$ maximizes $|T \cap S'|$. By Theorem \ref{thm:pivot}, $\calB(\calS) = \calB(\calS \setminus S \cup T)$, so the previous paragraph implies $T \in \calB^{\ast}(\calS \setminus S)$. Suppose there is some $b \in T \setminus S'$. Definition \ref{def:matroid} implies there is some $a \in S' \setminus T$ such that
\begin{displaymath}
([n] \setminus (T \setminus b \cup a)) \in \calB^{\ast}(\calS \setminus S).
\end{displaymath}
\noindent
Then, the bipartite graph $\Gamma_{\calS \setminus S}$ has an alternating path
\begin{displaymath}
b = a_0 - S_{\alpha_0} - a_1 - S_{\alpha_1} - \cdots - S_{\alpha_m} - a_{m+1} = a,
\end{displaymath}
\noindent
with $a_{\ell}, a_{\ell+1} \in S_{\alpha_{\ell}}$ for $0 \leq i \leq m$. Let $T_{m+1} = T \setminus b \cup a$, and for $m \geq i \geq 0$ define
\begin{displaymath}
T_{\ell} = T_{\ell+1} \setminus a_{\ell+1} \cup a_{\ell}.
\end{displaymath}
So, $T_{\ell}$ is obtained from $T_{\ell+1}$ via a pivot using the set $S_{\alpha_{\ell}}$, and $T_0 = T$. Then, $T \setminus b \cup a$ is reachable from $S$ by a series of pivots, contradicting the assumption that $T$ was chosen to maximize $|T \cap S'|$. Thus, $S'$ must be reachable from $S$ by a series of pivots.
\end{proof}

\begin{remark}
Ardila and Ruiz give a method of pivoting between presentations of a transversal matroid in \cite{ardila:cotransversal_matroid}. In Lemma 4.4, they prove that any two presentations of a transversal matroid are connected by their pivots. The pivoting procedure of Definition \ref{def:pivot} differs in that it only passes between minimal presentations of a transversal matroid, while Ardila and Ruiz's might use other presentations. We conjecture that all minimal presentations of a transversal matroid are connected by the pivots of Definition \ref{def:pivot}. 
\end{remark}

The following lemmas show that if $\calB(\calS)$ is a positroid and $\calS$ is a crossing minimal presentation, then there is an exact subsystem which may be used to perform a pivot removing the crossing.

\begin{lemma} \label{lem:small_excluded_minors}
Suppose $\calS$ is a minimal presentation of $\calB(\calS)$ and that $S_i$ crosses $S_j$ in $\calS$ with $a, c \in S_i$ and $b, d \in S_j$ witnessing this crossing. If there is a matching of size $k-2$ in $\Gamma_{\calS}$ from $\calS \setminus \{S_i, S_j\}$ to $[n] \setminus (S_j \cup \{a,c\})$, then $\dim(L(\calS)) > \dim(L(\calS) \cap Gr_{\geq 0}(k,n))$. Hence, $\calB(\calS)$ is not a positroid.
\end{lemma}

\begin{proof}
Suppose there is a matching in $\Gamma_{\calS}$ from $\calS \setminus \{S_i, S_j\}$ to $I$, for some $I \in \binom{[n]\setminus S_j, a, c}{k-2}$. Let $M_{\calS}(\mathbf{y})$ be some matrix obtained by evaluating the indeterminate entries of $M_{\calS}(\mathbf{x})$ at nonzero real values $y_{ij}$. We may suppose $j = 1$ and that the columns of $M$ are cyclically rotated so that $b$ is the first column. Let $m = | [b,d] \cap I |$. For $R \subset \calS$ and $C \subset [n]$ with $|R| = |C|$, let $M_{R,C}$ denote the determinant of the square submatrix of $M$ with rows $R$ and columns $C$. Since there is a matching from $\calS \setminus \{S_i,S_j\}$ to $I$, $M_{\calS\setminus S_1, I \cup a}, M_{\calS\setminus S_1, I \cup c} \neq 0$ at a generic evaluation of the $x_{ij}$. Then,
\begin{displaymath}
\mathrm{sgn}(M_{\calS,I \cup a,b}) = \mathrm{sgn}(y_{1b}) \mathrm{sgn}(M_{\calS\setminus S_1, I \cup a}),
\end{displaymath}
\noindent
and
\begin{displaymath}
\mathrm{sgn}(M_{\calS,I \cup a,d}) = (-1)^{m+1} \mathrm{sgn}(y_{1d}) \mathrm{sgn}(M_{\calS\setminus S_j, I \cup a}).
\end{displaymath}
So, if $M_{\calS}(\mathbf{y})$ represents a point in $Gr_{\geq 0}(k,n)$, $\mathrm{sgn}(y_{1b}) = (-1)^{m+1} \mathrm{sgn}(y_{1d})$. Further,
\begin{displaymath}
\mathrm{sgn}(M_{\calS,I \cup c,b}) = \mathrm{sgn}(y_{1b}) \mathrm{sgn}(M_{\calS\setminus S_1, I \cup c}),
\end{displaymath}
\noindent
and
\begin{displaymath}
\mathrm{sgn}(M_{\calS,I \cup c,d}) = (-1)^{m} \mathrm{sgn}(y_{1d}) \mathrm{sgn}(M_{\calS\setminus S_1, I \cup c}).
\end{displaymath}
So, if $M_{\calS}(\mathbf{y})$ represents a point in $\Gr_{\geq 0}(k,n)$, $\mathrm{sgn}(y_{1a}) = (-1)^{m} \mathrm{sgn}(y_{1c})$. Thus, $L(\calS)$ cannot intersect the positive Grassmannian in its full dimension. If $\calB(\calS)$ is a positroid, Theorem \ref{thm:loci_are_positroids} and point (iv) of Theorem \ref{thm:facts_about_marsh_reitsh_cell} imply
\begin{displaymath}
\dim(L(\calS) \cap \Gr_{\geq 0}(k,n)) = \dim(V_{\calB} \cap \Gr_{\geq 0}(k,n)) = \dim(V_{\calB}).
\end{displaymath}
\noindent
So, $\calB(\calS)$ cannot be a positroid.
\end{proof}

\begin{remark}
The asymmetry in Definition \ref{def:crossing_sets} is crucial in this lemma. If $a, b, c, d$ witness a crossing of $S_i$ and $S_j$ with $a, c, d \in S_i$, the sign computations above do not necessarily hold exchanging the roles of $S_i$ and $S_j$. In fact, $\Gamma_{\calS}$ may have a matching of size $k-2$ from $\calS \setminus \{S_i, S_j\}$ to $[n] \setminus (S_j \cup \{a,c\})$. Consider the set system $\calS$ consisting of
\begin{displaymath}
\begin{array}{rcl}
S_1 & = & \{1,3,5\}, \\
S_2 & = & \{2,3,4\}, \\
S_3 & = & \{2,4,5\}.
\end{array}
\end{displaymath}
\noindent
$\calB(\calS)$ is a positroid. However, $S_2$ crosses $S_1$ with $4, 1, 2, 3$ witnessing this crossing. So, Lemma \ref{lem:small_excluded_minors} implies there is not a matching from $S_3$ to $[5] \setminus S_1 \cup \{2,4\}$ in $\Gamma_{\calS}$. There is however a matching from $S_3$ to $[5]\setminus S_2 \cup \{1,3\}$ in $\Gamma_{\calS}$.
\end{remark}

\begin{lemma} \label{lem:finding_exact_subset}
Let $\calS$ be a minimal presentation of the positroid $\calB(\calS)$ and let $S_i$ cross $S_j$ in $\calS$. Then, there is an exact $\calT \subseteq \calS$ containing $S_i$ and $S_j$ such that $S_j$ is pivotable.\end{lemma}

\begin{proof}
Suppose that $S_i$ crosses $S_j$, and let $a, c \in S_i$ and $b, d \in S_j$ be elements witnessing this crossing. Suppose toward contradiction that there is no exact subsystem $\calT$ of $\calS$ such that $S_j$ is pivotable. Then, for all $\calT \subseteq \calS \setminus S_j$,
\begin{displaymath}
\left| \bigcup_{T \in \calT \cup S_j} T \right| > |\calT| + |S_j|.
\end{displaymath}

So, for all $x \in [n] \setminus S_j$,
\begin{equation} \label{eqn:strict_halls_inequality}
\left| \bigcup_{T \in \calT} T \setminus (S_i \cup x) \right| \geq |\calT|.
\end{equation}

Since $\calB(\calS)$ is a positroid, Lemma \ref{lem:small_excluded_minors} says there cannot be a matching in $\Gamma_{\calS}$ from $\calS \setminus \{S_i, S_j\}$ to $[n] \setminus (S_j \cup ac)$. So, Hall's Matching Theorem guarentees some $\calT \subseteq \calS \setminus \{S_i,S_j\}$ such that
\begin{displaymath}
\left| \bigcup_{T \in \calT} T \setminus (S_j \cup ac) \right| < |\calT|.
\end{displaymath}
\noindent
From (\ref{eqn:strict_halls_inequality}), we must have $\left| \bigcup_{T \in \calT} T \setminus (S_j \cup ac) \right| = |\calT| - 1$, and $a,c \in  \bigcup_{T \in \calT} T$.

If $S_i \setminus S_j \subset \bigcup_{T \in \calT} T$, then
\begin{displaymath}
\left| \bigcup_{T \in \calT \cup S_i \cup S_j} T \right| \leq |\calT| + |S_j| + 1.
\end{displaymath}
\noindent
Since $\calS$ is a minimal presentation, this inequality must hold with equality, and $\calT \cup S_i \cup S_j$ is an exact subsystem where $S_j$ is pivotable.

Otherwise, there is some $e \in S_i \setminus S_j$ not in $\bigcup_{T \in \calT} T$. Suppose without loss of generality that $e$ is on the same side of the chord from $b$ to $d$ as $c$. So, $a,b,e,d$ witnesses the fact that $S_i$ crosses $S_j$. Using Lemma \ref{lem:small_excluded_minors} and Hall's Theorem, there is some $\calT'$ such that
\begin{displaymath}
\left| \bigcup_{T \in \calT'} T \setminus (S_j \cup ae) \right| < | \calT'|.
\end{displaymath}
\noindent
Using (\ref{eqn:strict_halls_inequality}) as before,  $\left| \bigcup_{T \in \calT'} T \setminus S_j \right| = |\calT' | + 1$ and $a,e \in \bigcup_{T \in \calT'} T$.

Note that
\begin{displaymath}
\left| \left( \bigcup_{T \in \calT} T \setminus S_j \right) \cap \left( \bigcup_{T \in \calT;} T \setminus S_j \right) \right| \geq 1,
\end{displaymath}
\noindent
since $a$ is in this intersection. So,
\begin{displaymath}
\left| \bigcup_{T \in \calT \cup \calT'} T \setminus S_i \right| \leq |\calT \cup \calT'| + 1.
\end{displaymath}

Now, $a, c, e \in \bigcup_{T \in \calT \cup \calT'} T$. Continuing inductively, we may take $\calT$ to be some set such that $\left| \bigcup_{T \in \calT} T \setminus S_j \right| \leq |\calT|+1$ and $S_i \setminus S_j \subset \bigcup_{T \in \calT} T$. Then, $\calT \cup S_i \cup S_j$ is an exact subsystem where $S_j$ is pivotable.
\end{proof}

We call a presentation $\calS = \{S_1, S_2, \dots, S_k\}$ of $\calB$ {\textit{Gale minimal}} if it is a minimal presentation and $\calB$ has no other minimal presentation $\calS' = \{S'_1, S'_2, \dots, S'_k\}$ such that $|S'_i| = |S_i|$ and $S'_i \leq S_i$ in Gale order for all $1 \leq i \leq k$. Suppose the sets in $\calS$ are indexed such that $|S_i| \leq |S_{i+1}|$ for $1 \leq i \leq k-1$. A Gale minimal presentation of $\calB(\calS)$ may be produced algorithmically by, for each $1 \leq i \leq k$: (i) Identifying the maximal (with respect to containment) exact subsystem $\calT$ containing $S_i$ as a set of maximal size, then (ii) replacing $S_i$ with the Gale minimal basis of $\calB^{\ast}(\calT \setminus S_i)|_{\cup_{T \in \calT} T}$.

\begin{proof}[Proof of Proposition \ref{prop:all_sets_same_size}]
Let $\calS$ be a minimal presentation of $\calB(\calS)$. Suppose that $\calB(\calS)$ is a positroid and that all sets in $\calS$ have the same size. We may suppose that $\calB(\calS)$ is connected, since $\calB(\calS)$ has a noncrossing presentation if and only if its connected components form a noncrossing partition and each component has a noncrossing presentation.

Let $\calS$ be a Gale minimal presentation of $\calB(\calS)$. We claim $\calS$ is noncrossing. Suppose the set $S$ crosses the set $S'$. Lemma \ref{lem:finding_exact_subset} guarantees the existence of some exact subsystem $\calT = \{T_1, T_2, \dots, T_{|\calT|}\}$ containing $S$ and $S'$. Since $\calS$ is a Gale minimal presentation of $\calB(\calS)$, $\calT$ will be a Gale minimal presentation of $\calB(\calT)$. Let
\begin{displaymath}
\{t_1, t_2, \dots, t_{|S| + |\calT| - 1}\} = \bigcup_{T \in \calT} T,
\end{displaymath}
\noindent
where $t_1 < t_2 < \cdots < t_{|S| + |\calT| - 1}$. After possible reindexing the sets in $\calT$, the unique Gale minimal presentation of $\calT$ is
\begin{displaymath}
\begin{array}{rcl}
T_1 & = & \{t_1, t_2, \dots, t_{|S|-1}, t_{|S|}\}, \\
T_2 & = & \{t_1, t_2, \dots, t_{|S|-1}, t_{|S|+1}\}, \\
T_3 & = & \{t_1, t_2, \dots, t_{|S|-1}, t_{|S|+2}\}, \\
 & \vdots &  \\
T_{|\calT|} & = & \{t_1, t_2, \dots, t_{|S|-1}, t_{|S|+|\calT|-1}\}.
\end{array}
\end{displaymath}
\noindent
This presentation is noncrossing, violating that assumption that $S$ and $S'$ crossed.
\end{proof}

Unfortunately, Gale minimal presentations of transversal positroids may in general feature crossings. However, we may consider presentations which are minimal in $a$-Gale order for some $a$; the cyclic shift of Gale order obtained by using $<_a$ in place of the usual order $[n]$. So, if $I = \{i_1, i_2, \dots, i_k\}$ and $J = \{j_1, j_2, \dots, j_k\}$ with $i_1 <_a i_2 <_a \cdots <_a i_k$ and $j_1 <_a j_2 <_a \cdots <_a j_k$, then $I \leq_a J$ if and only if $i_{\ell} \leq_a j_{\ell}$ for each $1 \leq \ell \leq k$.

The following example illustrates a matroid $\calB(\calS)$ whose Gale minimal presentation is crossing, but which has a noncrossing $a$-Gale minimal presentation for some $a$. Conjecture \ref{conj:refined} states that this phenomena holds in general; that at least one of the $a$-Gale minimal presentations of a transversal positroid will always be noncrossing.

\begin{example} \label{ex:minimal_doesnt_always_work}
Let $\calS$ be the set system consisting of
\begin{displaymath}
\begin{array}{rcl}
S_1 & = & \{1,2,3,4\}, \\
S_2 & = & \{1,2,3,5\}, \\
S_3 & = & \{4,5,6\}.
\end{array}
\end{displaymath}
\noindent
The matroid $\calB(\calS)$ is a positroid, and this presentation is Gale minimal. However, $6,1,4,5$ witnesses a crossing of $S_3$ and $S_2$. A minimal presentation in $4$-Gale order is $\{4512, 4513, 456\}$. This presentation is noncrossing.
\end{example} 

\begin{conjecture} \label{conj:refined}
Let $\calB$ be a transversal positroid. There is a minimal presentation $\calS$ of $\calB$ which is minimal in $a$-Gale order for some $a$, and which is noncrossing.
\end{conjecture}

This conjecture holds in the cases described in Propositions \ref{prop:richardsons_behave} and \ref{prop:all_sets_same_size} and has been verified exhaustively for matroids of rank up to $4$ on up to $10$ elements. Additionally, it has received extensive computational verification on  randomized examples of matroids of rank up to $8$ on up to $14$ elements.

\section{Comparison with Other Structures} \label{sec:comparison_with_other_classes}

We note that not all positroids are of the form $\calB(\calS)$ for some $\calS$.

\begin{example} \label{ex:non_transversal_positroid}
Consider the matrix
\begin{displaymath}
\left(
\begin{array}{cccccc}
1 & 1 & 0 & 0 & -1 & -1 \\
0 & 0 & 1 & 1 & 1 & 1
\end{array}
\right).
\end{displaymath}
\noindent
All maximal minors of this matrix are nonnegative, so the matroid $\calB$ it represents is a positroid. However, this matroid is not a transversal matroid. Suppose that $\calB = \calB(\{S_1, S_2\})$. Suppose that $1 \in S_1$. Then, $2 \notin S_2$, since $12\notin \calB$. So, $2 \in S_1$, and $1 \notin S_2$. Then, $3,4,5,6 \in S_2$, since $13, 14, 15, 16 \in \calB$. Then, since $34, 56 \notin \calB$, $3,4,5,6 \notin S_1$. So, $35 \notin \calB(\{S_1, S_2\})$, contradicting the fact that $\calB = \calB(\{S_1, S_2\})$.
\end{example}

The {\textit{interval rank function}} the map sending a $k \times n$ matrix $M$ to the $n \times n$ upper triangular matrix $r(M)$, where
\begin{displaymath}
r(M)_{ij} = \mathrm{rank}(\mbox{the submatrix of $M$ using columns $\{i, i+1,\dots, j\}$}).
\end{displaymath}

Note that the intervals appearing here are ordinary intervals, not cyclic intervals. An {\textit{interval positroid variety}} is the set of points in $\Gr(k,n)$ with a fixed interval rank matrix. The matroid represented by a generic point in an interval positroid variety is an {\textit{interval positroid}}. Interval positroid varieties were introduced by Knutson in \cite{knutson:interval_positroid} to study the degenerations appearing in Vakil's ``geometric Littlewood-Richardson rule," \cite{vakil:littlewood_richardson}. Interval positroid varieties are positroid varieties. All Schubert varieties, opposite Schubert varieties, and Richardson varieties are interval positroid varieties.  The matroid from Example \ref{ex:non_transversal_positroid} is an example of an interval positroid which is not a transversal matroid. The following example provides a transversal positroid which is not an interval positroid.

\begin{example}
Let $\calS = \{1245, 23, 56\}$. The set system $\calS$ is noncrossing, so Theorem \ref{thm:loci_are_positroids} implies $\calB(\calS)$ is a positroid. The interval rank matrix of a generic point in $L(\calS)$ is
\begin{equation} \label{eqn:interval_rank_matrix}
\left(
\begin{array}{cccccc}
1 & 2 & 2 & 2 & 3 & 3 \\
0 & 1 & 2 & 2 & 3 & 3 \\
0 & 0 & 1 & 2 & 3 & 3 \\
0 & 0 & 0 & 1 & 2 & 2 \\
0 & 0 & 0 & 0 & 1 & 2 \\
0 & 0 & 0 & 0 & 0 & 1
\end{array}
\right).
\end{equation}
The interval positroid variety associated to (\ref{eqn:interval_rank_matrix}) is the smallest interval positroid variety containing $L(\calS)$. So, if $\calB(\calS)$ is an interval positroid, it must be the one defined by (\ref{eqn:interval_rank_matrix}).
Computing from (\ref{eqn:interval_rank_matrix}) a bounded affine permutation as described in \cite{knutson:interval_positroid}, then associating a positroid to this bounded affine permutation as described in \cite{knutson:juggling}, the interval positroid associated to (\ref{eqn:interval_rank_matrix}) is
\begin{displaymath}
\{125, 126, 135, 136, 145, 146, 156, 235, 236, 245, 246, 256, 345, 346, 356\}.
\end{displaymath}
\noindent
Notably, $145$ is a basis of this interval positroid, but $145 \notin \calB(\calS)$.
\end{example}

Even though $\MR(\calB) \cap L(\calS)$ is dense in both $\MR(\calB)$ and $L(\calS)$, neither set is in general contained in the other.

\begin{example}
Let $\calS = \{134,234\}$. Then,
\begin{displaymath}
\mathrm{span}\left(
\begin{array}{cccc}
1 & 0 & -1 & -1 \\
0 & 1 & 1 & 1
\end{array}
\right) \in L(\calS).
\end{displaymath}
\noindent
This point is in $\Gr_{\geq 0}(2,4)$, but represents a matroid aside from $\calB(\calS)$. So, this point is not in $\MR(\calB(\calS))$.
\end{example}

\begin{example}
The Marsh-Rietsch cell associated to the positroid $\calB = \binom{[4]}{2}$ is
\begin{displaymath}
\left\{ \mathrm{span}\left(
\begin{array}{cccc}
1 & 0 & -a_3 & -(a_3 a_4 + a_3 a_2) \\
0 & 1 & a_1 & a_1 a_2
\end{array}
\right) \ : \ a_1, a_2, a_3, a_4 \in \rr^{\ast}
\right\}.
\end{displaymath} 
The set $\calS = \{134, 234\}$ satisfies $\calB(\calS) = \calB$. The subset of $\MR(\calB)$ where $a_4 = -a_2$ is not contained in $L(\calS)$.
\end{example}

A {\textit{rank variety}} is $\overline{L(\calS)}$ where each $S \in \calS$ is an (ordinary) interval. Rank varieties were introduced by Billey and Coskun in \cite{billey:generalized_richardson_varieties} as a generalization of Richardson varieties.

Diagram varieties were introduced by Liu in \cite{liu:diagram_varieties}, and studied by Pawloski in \cite{pawlowski:interval_positroid_varieties}. They define a {\textit{diagram}} to be a subset $D$ of $[k] \times [n-k]$. If $A$ is a $k \times n$ matrix, let $[A \mid I_k]$ be the matrix obtained by appending a $k \times k$ identity matrix to the right of $A$. The {\textit{diagram variety}} defined by $D$ is the closure of
\begin{displaymath}
\{\mathrm{span}[A \mid I_k] : A \in M_{k,n-k} \mbox{ with } A_{ij} = 0 \mbox{ when } (i,j) \in D\}.
\end{displaymath}
\noindent
Evidently, all diagram varieties are closures of basis shape loci, but there are basis shape loci whose closures are not diagram varieties.

\section{Dominos and Wilson Loops} \label{sec:applications_of_basis_shape_loci}

An amplituhedra is a projection of $\Gr_{\geq 0}(k,n)$ to $\Gr(k,k+m)$ for some $m$ by a totally positive matrix. When $m = 4$, volumes of amplituhedra conjecturally compute scattering amplitudes in $\calN = 4$ supersymmetric Yang-Mills theory (SYM). Toward verifying this conjecture, Arkani-Hamed and Trnka conjecture in \cite{arkani:amplituhedron} that a collection of positroid cells called BCFW cells project to a triangulation of the $m = 4$ amplituhedron, in the sense that their images are dense in the amplituhedron and overlap in a set of measure zero. BCFW cells, after Britto, Cachazo, Feng, and Witten, arise from the BCFW recurrence relation, which is known to compute certain amplitudes in $\calN = 4$ SYM \cite{britto:bcfw_recurrence}.

Karp, Williams, and Zhang provide a program for proving that BCFW cells triangulate amplituhedra in \cite{karp:decompositions}, and for finding collections of positroid cells triangulating amplituhedra for other even $m$. Roughly, their strategy is to find a basis of a special shape, called a domino basis, for any plane in the positroid cell under consideration, then to apply sign variation techniques to these basis shapes to verify disjointness of these cells' projections. Conjecture A.7 in \cite{karp:decompositions} says that points in BCFW cells admit domino bases. Their sign variation techniques are similar to those used in \cite{arkani:unwinding_the_amplituhedron} to describe amplituhedra in terms of binary codes. When $k = 1$, the conjectured triangulation of \cite{karp:decompositions} is among the triangulations of cyclic polytopes from Theorem 4.2 in \cite{rambau:cyclic_polytopes}.

Part of the impetus for this work was to serve this program of Karp, Williams, and Zhang. Currently, there is a sense of the kind of basis shapes which should be amenable to their sign variation arguments and in many cases of the cells which should appear in a triangulation of the amplituhedra. One missing component is a formal way of connecting these two ideas. The hope is that, rather than first identifying a family of positroid cells and then trying to find special bases for points in these cells, one might be able to first identify the sorts of basis shapes $\calS$ which are amenable to sign variation arguments, then work with positroid cells $L(\calS) \cap \Gr_{\geq 0}(k,n)$. Since $\overline{L(\calS)} = \overline{V_{\calB(\calS)}}$, working with $L(\calS)$ is no different than working directly with the positroid cell from the perspective of producing a triangulation.

Briefly, we introduce the basis shapes appearing in Karp, Williams, and Zhang's program which our work is presently able to handle; this is just a subset of the basis shapes Karp, Williams, and Zhang consider. Say that a vector $\mathbf{v}$ is a {\textit{$i$-domino}} if $\mathrm{supp}(\mathbf{v}) = \{i,i+1\}$, where by convention $n+1 = 1$\footnote{\cite{karp:decompositions} treats $n$-dominos slightly differently, defining $\mathbf{v}$ to be an $n$-domino if $\mathrm{supp}(\mathbf{v}) = n$. We choose our convention since it will more natural for drawing a connection with Wilson loop cells.}. For $I \subset [n]$, say $\mathbf{v}$ is an {\textit{$I$-domino}} if $\mathbf{v}$ is a sum of $i$-dominos with disjoint supports for all $i \in I$. For $\calI = \{I_1, \dots, I_k\}$, say $V \in \Gr(k,n)$ admits an {\textit{$\calI$-domino basis}} if $V$ is the span of $I$-dominos for $I \in \calI$. Given such an $\calI$, let
\begin{equation} \label{eqn:domino_to_set_family}
\calI' = \{I'_1, \dots, I'_k\},
\end{equation}
\noindent
where
\begin{equation} \label{eqn:dominos_to_sets}
I'_j = I_j \cup \{i+1 : i \in I_j\}.
\end{equation}
\noindent
Evidently, the set of planes admitting $\calI$-domino bases is exactly $L(\calI')$.

In \cite{karp:decompositions}, $i$-dominos are further required to have their adjacent entries have the same sign. This requirement is a consequence of positivity and the shape constraints.

\begin{proposition}
Let $V \in \Gr_{\geq 0}(k,n) \cap L(\calS)$ for some set system $\calS$ satisfying the hypotheses of Theorem \ref{thm:dimension_and_minimal_presentations}. Let $\mathbf{v}_{S_1}, \dots, \mathbf{v}_{S_k}$ be a basis of shape $\calS$ for $V$ and let $(v_1, \dots, v_n) = \mathbf{v}_{S_1}$. If $v_i, v_{i+1} \neq 0$, then
\begin{displaymath}
\mathrm{sgn}(v_{i+1}) =
\begin{cases} \mathrm{sgn}(v_{i}) & \mbox{if } i \neq n, \\
(-1)^{k-1} \mathrm{sgn}(v_{i}) & \mbox{if } i = n.
\end{cases}
\end{displaymath}
\end{proposition}

\begin{proof}
From point (iii) of Theorem \ref{thm:dimension_and_minimal_presentations}, there is some $J = \{j_1, \dots, j_{k-1}\} \subseteq [n] \setminus S_1$ such that there is a matching from $S_2, \dots, S_k$ to $J$ in $\Gamma_{\calS}$. Let $A = (a_{\ell,m})$, where $a_{\ell,m}$ is the $j_m^{th}$ coordinate of $\mathbf{v}_{S_{\ell+1}}$. Then,
\begin{displaymath}
\Delta_{J \cup i} = v_i \det(A),
\end{displaymath}
\noindent
and
\begin{displaymath}
\Delta_{J \cup i+1} =
\begin{cases}
v_{i+1} \det(A) & \mbox{if } i \neq n, \\
(-1)^{k-1} v_{i+1} \det(A) & \mbox{if } i = n.
\end{cases}
\end{displaymath}
Since $V \in \Gr_{\geq 0}(k,n)$, these two Pl\"ucker coordinates have the same sign.
\end{proof}

These are only the simplest basis shapes appearing in Karp, Williams, and Zhang's work. There are still several gaps to be filled before this present work can be useful for describing all basis shapes appearing in their work. Notably, the results in this paper apply in the case when the entries of the matrix $M_{\calS}(\mathbf{x})$ are all independent. To capture the basis shapes appearing in \cite{karp:decompositions}, one would need to extend our results to describe cases where some of the entries of $M_{\calS}(\mathbf{x})$ are prescribed to be equal to each other. Another useful feature would be a way of translating between the set system $\calS$ and any of the many combinatorial objects indexing positroids described in \cite{postnikov:total_positivity}.

Another motivation for the present work comes from another program for computing amplitudes in $\calN = 4$ SYM via Wilson loop diagrams (also called MHV diagrams). This program is introduced from a physical perspective in \cite{bullimore:mhv_diagrams} and surveyed in a way more accessible to mathematicians in \cite{agarwala:wilson_loop_diagram}. The geometric spaces arising in this program are basis shape loci $L(\calS)$ for a particular class of $\calS$ defined by Wilson loop diagrams. One goal was to illustrate a connection between these shapes and domino bases. Corollary \ref{cor:wld_dominos} shows that the set of points in $\Gr_{\geq 0}(k,n)$ admitting $\calI$-domino bases where $|I| = 2$ for all $I \in \calI$ is exactly the union of Wilson loop cells.

\begin{definition} \label{def:wilson_loop_diagram}
A {\textit{Wilson loop diagram}} is a set of unordered pairs $\calP \subset \binom{[n]}{2}$ such that for all $P \in \calP$, if $i \in P$, then $i+1 \notin P$.\end{definition}

Each $P \in \calP$ is called a {\textit{propagator}}. Graphically, one commonly represents a Wilson loop diagram by a convex polygon whose vertices are labeled by the elements of $[n]$ counterclockwise. For each $P = \{i_P, j_P\} \in \calP$, draw an internal wavy line between the edges of the polygon defined by the vertices $\{i_P, i_P +1\}$ and $\{j_P, j_P +1\}$. For example, the Wilson loop diagram $\{ 24, 46\}$ would have the following graphical representation.

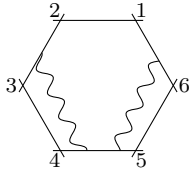
\begin{displaymath}
\begin{tikzpicture}
\drawWLD{6}{1}
	\drawnumbers
	\drawprop{6}{0}{4}{1}
	\drawprop{2}{0}{4}{-1}
\end{tikzpicture}
\end{displaymath}

Given a Wilson loop diagram $\calP = \{P_1, P_2, \dots, P_k\}$, define $\calP'$ and $P'_i$ as in (\ref{eqn:domino_to_set_family}) and (\ref{eqn:dominos_to_sets}). The basis shape locus $L(\calP')$ is called the {\textit{Wilson loop cell}} associated to the diagram.

\begin{definition} \label{def:admissible_wld}
A Wilson loop diagram $\calP$ is {\textit{admissible}} if the following hold:
\begin{itemize}
\item[(i)] for all $\calQ \subseteq \calP$, $\left| \bigcup_{Q \in \calQ} Q' \right| \geq |\calQ| + 3$, and
\item[(ii)] if $P = (i_P, j_P)$, $Q = (i_Q, j_Q) \in \calP$ are two propagators, then $i_P < i_Q < j_Q < j_P$ in the cyclic ordering of $[n]$.
\end{itemize}
\end{definition}

Let $\calP$ be an admissible Wilson loop diagram and $\calP'$ be the associated set system. Theorem 3.38 in \cite{agarwala:wilson_loop_positroid} says that $\calB(\calP')$ is a positroid. One would like to be able to apply tools from the positroid literature to study Wilson loop diagrams. However, the object of interest is really the cell $L(\calP')$, not the matroid $\calB(\calP')$. In the literature, it was understood and used, but not clear that working with the positroid cell associated to $\calB(\calP')$ was, up to a set of measure zero, equivalent to working with the Wilson loop cell $L(\calP')$. This equivalence is a corollary of Theorems \ref{thm:loci_are_positroids} and \ref{thm:no_crossing_implies_positroid}.

\begin{theorem} \label{thm:consequences_for_wilson_loops}
Let $\calP$ be an admissible Wilson loop diagram, and $\calP'$ be the associated set system. Then,
\begin{itemize}
\item[(i)] $\dim(L(\calP')) = 3k$.
\item[(ii)] The matroid $\calB(\calP')$ is a positroid.
\item[(iii)] $\overline{L(\calP')} = \overline{V_{\calB(\calP')}}$.
\end{itemize}
\end{theorem}

\begin{proof}
Point (i) follows from point (i) of Theorem \ref{thm:dimension_and_minimal_presentations}. Point (ii) follows from the fact $\calI'$ is noncrossing and Theorem \ref{thm:no_crossing_implies_positroid}. Point (ii) originally appears as Theorem 3.38 in \cite{agarwala:wilson_loop_positroid}. Point (iii) follows from Theorem \ref{thm:loci_are_positroids}.
\end{proof}

Wilson loop diagrams come equipped with a notion of exact subdiagrams, analogous to Definition \ref{lem:contracting_to_exact_subdiagram}.

\begin{definition} \label{def:wld_exact_subdiagram}
Let $\calP$ be a Wilson loop diagram satisfying point (i) of Definition \ref{def:admissible_wld}, and let $\calP'$ be the associated set system. A subset $\calQ \subseteq \calP$ is an {\textit{exact subdiagram}} if its associated set system is an exact subsystem of $\calP'$ in the sense of Definition \ref{def:exact_subset}. That is,
\begin{displaymath}
\left| \bigcup_{Q \in \calQ} Q' \right| = |\calQ| + 3.
\end{displaymath}
\end{definition}

Let $\calP$ and $\calQ$ be Wilson loop diagrams satisfying point (i) of Definition \ref{def:admissible_wld} and let $\calP'$ and $\calQ'$ be their associated set systems. Say that $\calP \sim \calQ$ if $\calB(\calP') = \calB(\calQ')$. Since all sets in $\calP'$ have the same size, any set in an exact subsystem of $\calP'$ may be pivoted in the sense of Definition \ref{def:pivot}. Further, any exact subsystem supported on $E \subseteq [n]$ must be defined by $|E| - 3$ propagators from $\calP$. Using these observations, one can show that any exact two exact subdiagrams supported on the set $E$ may be pivoted to one another. Moreover, it is always possible to arrange $|E|-3$ propagators supported on $E$ vertices to obey both points (i) and (ii) of Definition \ref{def:admissible_wld} (in fact, the number of ways to do so is a Catalan number). We record these observations.

\begin{theorem}[Theorem 1.18 in \cite{agarwala:wilson_loop_positroid}] \label{thm:uncrossing_wlds}
Let $\calP$ and $\calQ$ be Wilson loop diagrams satisfying point (i) of Definition \ref{def:admissible_wld}. If $\calP$ and $\calQ$ differ by only an exact subdiagram supported on some $E \subseteq [n]$, then $\calP \sim \calQ$. If $\calP$ is exact, then $\calP \sim \calR$ for some $\calR$ satisfying both points (i) and (ii) of Definition \ref{def:admissible_wld}.
\end{theorem}

\begin{theorem} \label{thm:converse_to_wld_positroid}
Let $\calP$ be a Wilson loop diagram satisfying point (i) of Definition \ref{def:admissible_wld} and $\calP'$ be the associated set system. Then, $\calB(\calP')$ is a positroid if and only if $\calP \sim \calQ$ for some Wilson loop diagram $\calQ$ satisfying both points (i) and (ii) of Definition \ref{def:admissible_wld}.
\end{theorem}

\begin{proof}
Let $\calP$ be a Wilson loop diagram satisfying point (i) but not necessarily point (ii) of Definition \ref{def:admissible_wld} and let $\calP'$ its associated set systems. If $\calP \sim \calQ$ for some $\calQ$ satisfying both points (i) and (ii) of Definition \ref{def:admissible_wld}, then Theorem 3.38 in \cite{agarwala:wilson_loop_positroid} or Theorem \ref{thm:no_crossing_implies_positroid} implies that $\calB(\calP')$ is a positroid. If $\calB(\calP')$ is a positroid, but there is some pair of crossing propagators in $\calI$, Lemma \ref{lem:finding_exact_subset} says we can find some exact subdiagram involving these crossing propagators. Then, Theorem \ref{thm:uncrossing_wlds} says this exact subdiagram may be replaced with any noncrossing exact subdiagram supported on the same set of vertices. Repeatedly applying this argument, $\calP \sim \calQ$ for some $\calQ$ which is noncrossing.
\end{proof}

This theorem has a compelling rephrasing, using the language of dominos.

\begin{corollary} \label{cor:wld_dominos}
The set of points in $\Gr_{\geq 0}(k,n)$ admitting $\calI$-domino bases where $|I| = 2$ for all $I \in \calI$ is exactly
\begin{displaymath}
\bigcup_{\calP} L_{\geq 0}(\calP'),
\end{displaymath}
\noindent
where the union is across all admissible Wilson loop diagrams.
\end{corollary}

\end{document}